\documentclass[11pt]{amsart}
\usepackage{fullpage}

\usepackage{hyperref}
\usepackage{amsmath,amsfonts,amssymb}
\usepackage{mathrsfs}
\usepackage{verbatim}
\usepackage{amsthm}
\usepackage{mathrsfs}

\newtheorem{theorem}{Theorem}[section]
 \newtheorem{corollary}[theorem]{Corollary}
 \newtheorem{lemma}[theorem]{Lemma}
 \newtheorem{proposition}[theorem]{Proposition}
 \theoremstyle{definition}
 \newtheorem{definition}[theorem]{Definition}
 
 \theoremstyle{remark}
 \newtheorem{remark}[theorem]{Remark}
  \newtheorem{ex}[theorem]{Example}
 \numberwithin{equation}{section}

\def \bC {\mathbb C}

\def \bN {\mathbb N}

\def \bR {\mathbb R}

\def \bT {\mathbb T}

\def \bZ {\mathbb Z}

\def \cA {\mathcal A}

\def \cD {\mathcal D}

\def \cF {\mathcal F}

\def \cH {\mathcal H}

\def \cL {\mathcal L}
\def \cM {\mathcal M}

\def \cR {\mathcal R}
\def \cS {\mathcal S}

\def \fg {\mathfrak g}

\def \fv {\mathfrak v}

\def \fU {\mathfrak U}

\def \sL{\mathscr L}
\def \sE{\mathscr E}

\def \tr {{\rm Tr}}
\def \ad {{\rm ad}}
\def \id {\text{\rm I}}
\def \Op {{\rm Op}}
\def \supp{{{\rm supp}}}
\def \Gh {{\widehat G}}
\def \eps {\varepsilon}
\def \vol {{\rm vol}}

\def \sp {{\rm sp}}

\begin{document}
\author[V. Fischer]{V\'eronique Fischer}
\address[V. Fischer]%
{University of Bath, Department of Mathematical Sciences, Bath, BA2 7AY, UK} 
\email{v.c.m.fischer@bath.ac.uk}

\title{semiclassical analysis on compact nil-manifolds}



\maketitle

\begin{abstract}
In this paper, 
we define and study semi-classical analysis and semi-classical limits
on compact nil-manifolds. 
As an application, we obtain properties of quantum limits for sub-Laplacians in this context, and more generally for positive Rockland operators.  
\end{abstract}

\makeatletter
\renewcommand\l@subsection{\@tocline{2}{0pt}{3pc}{5pc}{}}
\makeatother

\tableofcontents

\section{Introduction}


The analysis of hypoelliptic operators  has made fundamental progress 
over the last twenty years, 
see e.g. \cite{androulidakis2022pseudodifferential,PongeAMS2008,vanErp} and references therein. 
The underlying methods and ideas are so comprehensive 
that they extend beyond the class of sub-Riemannian manifolds to the setting of H\"ormander's family of vector fields. 
In fact, a crucial tool for these results has turned out to be the generalisation of Connes' tangent groupoid \cite{androulidakis2022pseudodifferential,vanErp+Yuncken,Dave+Haller2020,Dave+Haller2022}, in some sense reuniting  the lifting theory of Stein and his collaborators (see, for example, \cite{RotschildStein} or \cite{Street}) with more geometric concepts, such as tangent groups and nilpotentization \cite{subR,Montgomery}. 
However, these approaches are seldom symbolic in a way that would make them appropriate for general questions in spectral geometry and micro-local and semi-classical analysis, such as quantum ergodicity.

The primary objective of this paper is to understand quantum limits 
for a large class of hypoelliptic operators in an accessible setting. 
This class encompasses intrinsic sub-Laplacians on compact nil-manifolds, without imposing any additional assumptions beyond the existence of a Carnot structure on the underlying group. 
The methods and results presented herein align with the symbolic approach of semi-classical analysis in sub-Riemannian and subelliptic settings and utilising the representation theory of nilpotent Lie groups, see \cite{bologna} and references therein. 
Following the group case studied in \cite{FFPisa,FFchina,FFJST}, 
the next natural context to test this approach
 is to explore the context of compact nil-manifolds
 \cite{Fermanian+Letrouit,withSoren}. 
 These manifolds serve as the analogues of the torus $\bT^n$ (the quotient of $\bR^n$ by a lattice) in spectral Euclidean geometry. 
 Already the (flat and commutative) tori  provide a rich framework for the conventional (i.e., Euclidean and commutative) semi-classical or micro-local analysis. For instance, references such as \cite{Nalini+Fab,Jakobson} offer valuable insights into this topic.
Moreover, reduced quantum ergodicity for the canonical Laplacian on the flat torus can be obtained with elementary means involving semi-classical analysis in this context, see \cite[Section 3.1]{Nalini+Clo+Faure} and Theorem \ref{thm_QETn} in this paper  for a precise statement. 
The starting point and initial motivation behind  this paper was to understand whether any form of quantum ergodicity  could be obtained on nil-manifolds in the same way. 
It seems the non-commutativity makes the situation much more complicated (see Remark \ref{rem_originalmotivation}, Sections \ref{subsubsec_obsgrHtype} and \ref{subsec_comments}).

    
The pseudo-differential theory presented in this paper relies on the Kohn-Nirenberg quantization, which is always available for groups satisfying Dixmier's Plancherel theorem (see e.g. \cite{LinoCLoMe}). 
This leads to replacing the conventional phase space in commutative settings with $G\times \Gh$, where $\Gh$ is the unitary dual of $G$. 
In such a highly non-commutative setting, scalar-valued functions on the cotangent space are no longer suitable symbols. Instead, 
we consider symbols that are fields of operators (or equivalently measurable sections) on $G\times \Gh$ or $M\times \Gh$.
Graded nilpotent Lie groups naturally possess dilations, which provide a semi-classical scaling on the group and its unitary dual, and therefore for the Kohn-Nirenberg quantization $\Op^{(\eps)}$.
The case of symbol classes of H\"ormander type was studied in \cite{withSoren} with applications to semiclassical Weyl laws. In this paper, we focus on the semiclassical analysis based on smoothing symbols, especially the study of  further asymptotics and semiclassical measures.

  
 \smallskip 
  
 The paper aims to develop the theory supporting semi-classical measures in this context, particularly focusing on quantum limits for sub-Laplacians and more generally positive Rockland operators.
In the Euclidean or Riemannian case, 
the existence of semi-classical or micro-local defect measures is traditionally proven using 
G\aa rding inequalities in pseudo-differential theories 
(for example, see \cite[Chapter 5]{zworski} and \cite{Gerard} respectively).
However, an alternative proof of the existence of Euclidean micro-local defect measures (or H-measures) relies solely on multiplication operators and Fourier multipliers (see \cite{Tartar}).
A potentially deeper and unifying argument, essentially due to Vladimir Georgescu, has emerged in recent years. This argument relies on the properties of states of $C^*$-algebras, particularly the $C^*$-algebra generated by the semi-classical pseudo-differential calculus.
The justification for the positivity  comes naturally from the symbolic calculus, the positivity of a state, and the fact that the limit of a sequence of states is itself a state.
The arguments and comments in the Euclidean case and in the context of nilpotent Lie groups can be found in the expository article \cite{bologna}.
In non-commutative settings, such as those provided by sub-Riemannian contexts or nilpotent Lie groups, we cannot expect any adapted notion of micro-local or semi-classical measures to be scalar-valued. Instead, they should be operator-valued, like the symbols.
Such a notion was introduced on nilpotent Lie groups in \cite{FFchina,FFPisa}, and further developed in \cite{FFJST} using the approach via states of $C^*$-algebras.
In this paper, we demonstrate that the same arguments extend to the nil-manifold setting:

\begin{theorem}
	\label{thm_scL}
	We consider $\Gamma$  a discrete co-compact subgroup of 
a graded nilpotent Lie group $G$, and  denote by $M:=\Gamma \backslash G$
the corresponding compact nil-manifold. 
Let $(\phi_\eps)_{\eps\in (0,1]}$ be a bounded family  in $L^2(M)$.
There exists a sequence $(\eps_k)_{k\in \bN}$ going to 0 as $k\to \infty$ and a positive operator-valued measure $\Gamma d\gamma \in {\cM}^+_{ov}(M\times \Gh)$
such that we have the convergence  for any $\sigma\in \cA_0$:
$$
\left(\Op^{(\eps)}(\sigma) \phi_\eps,\phi_\eps\right)_{L^2(M)}
\longrightarrow_{\eps=\eps_k, k\to \infty} \iint_{M\times \Gh} {\rm Tr}\left(\sigma(x,\pi) \Gamma(x, \pi)\right)d \gamma(x, \pi).
$$
Moreover, we may assume that the limit
$\lim_{k\to \infty} \|\phi_{\eps_k}\|_{L^2(M)}$ exists
and in this case, 
$$
\iint_{M\times \Gh} {\rm Tr}\left( \Gamma(x, \pi)\right)d \gamma(x, \pi) = \lim_{k\to \infty} \|\phi_{\eps_k}\|_{L^2(M)}^2.
$$
\end{theorem}

We call 
$\Gamma d\gamma$ the semi-classical measure of the family $(\phi_\eps)_\eps$
for the  sequence $(\eps_k)_{k\in \bN}$.
The class of smoothing symbols $\cA_0$ is defined in Section \ref{subsec_sclpdo}, and the concept of vector-valued measures is introduced in Section \ref{subsec_dualC*MGh}.
Theorem \ref{thm_scL} follows from the arguments outlined above (see Section \ref{sec_sclim}).
This notion of semi-classical measure provides a way to describe the obstruction to $L^2$-strong convergence of the family of functions $\phi_\eps$ on the `phase-space' $M\times \Gh$.
We further explore the subject by decomposing the semi-classical measures into a sum of a scalar-valued measure and an operator-valued measure, corresponding to the 1-dimensional and infinite-dimensional parts of the unitary dual, respectively, see \eqref{eq_GdgdecGh} in Section \ref{subsec_LinftyMGh}.
Additionally, we delve into the properties of \emph{quantum limits}, which, in this context, refer to semi-classical limits associated with a sequence of eigenfunctions of a sub-Laplacian or, more generally, a positive Rockland operator on $M$. 
This concept is further explored in Section \ref{sec_QLR}.
In particular, we demonstrate that the operator-valued measure possesses properties of localisation and invariance.
This work raises several intriguing questions. For instance, the analogous properties in more commutative settings enable the description of all quantum limits on e.g. tori and  spheres.
It would be fascinating to determine all quantum limits in the context of nil-manifolds of Heisenberg types and for  quotients of the Engel group. 
 
\medskip

In this paper, we consider not only 
intrinsic sub-Laplacians on Carnot  groups but more generally positive Rockland operators on graded Lie groups. 
This broader approach offers several advantages.
Firstly, by considering this larger class of hypoelliptic operators, we gain properties shared with elliptic differential operators in Euclidean and Riemannian settings. Specifically, a positive Rockland operator $\cR$ is not necessarily of degree 2, and any power $\cR^N$, where $N\in \bN$, is also a positive Rockland operator. In addition, 
the setting of  Rockland operators also lends itself to considering vector-valued  operators, such as the Rumin complex and the Rumin-Seshadri operators. 
Secondly, while the product of two stratified Lie groups may possess a stratified structure, it may be more advantageous to equip the resulting product with a graded structure (that is, with a Lie algebra equipped with a gradation but not necessarily generated by the first stratum). For instance, when studying a heat partial differential equation $\partial_t +\cL_0$ on a stratified group $G_0$, where $\cL_0$ is a sub-Laplacian on $G_0$, the corresponding differential operator on the product group $G=\bR\times G_0$ will not be homogeneous if $G$ is equipped with the stratified structure. However, it will be if $G$ is equipped with the graded structure derived from considering $\bR$ with weight 2.
Thirdly, the heat semi-groups associated with the sub-Laplacians on stratified groups are Markov, which implies, in particular, the positivity of the heat kernels. On a Lie group, Hunt's theorem states that the left-invariant semi-group generated by a left-invariant differential operator of degree greater than 2 will not be Markov. However, it turns out that most arguments involving heat kernels in our field of study do not rely on Markov theory but rather on hypoellipticity as a more fundamental premise.
Lastly, our choice of setting with Rockland operators on graded groups serves as a means to emphasise the significance of hypoellipticity. 
This is the setting of the Helffer-Nourrigat theorem, which characterises hypoelliptic left-invariant operators on nilpotent Lie groups. Recently, it has been extended to manifolds \cite{androulidakis2022pseudodifferential} that are not necessarily subRiemannian or operators that are not necessarily sub-Laplacians.
Finally, we note that many results and techniques presented in this paper may generalise to positive Rockland operators perturbed by lower-order terms, such as a potential (see, for example, \cite{withSoren}). 

\medskip

The paper is organised as follows.
After presenting the setting in detail in Sections \ref{sec_prelM} and \ref{secR}, we develop the semi-classical calculus based on representation theory in Section \ref{sec_sccM}. This provides us with the tools to determine asymptotics in Section \ref{sec_asympt}.
In Section \ref{sec_sclim}, we define the notion of semi-classical limit of a sequence of functions in $L^2(M)$ in our context.
In Section \ref{sec_QLR}, we consider the semi-classical limit of a sequence of eigenfunctions of a positive Rockland operator and demonstrate some properties of localisation and invariance.

\subsection*{Acknowledgement}
The author expresses gratitude to the Leverhulme Trust for their financial support through the Research Project Grant 2020-037. Additionally, she is thankful to the anonymous referees for their valuable comments that significantly contributed to the improvement of the paper.

\section{Preliminaries on nilpotent Lie groups and compact  nil-manifold} 
\label{sec_prelM}

In this section, we set our notation for nilpotent Lie groups and nil-manifolds.
We also recall some elements of harmonic analysis in this setting. 

\subsection{About nilpotent Lie groups}
\label{subsec_aboutGnilp}
In this paper, a nilpotent Lie group $G$ is always assumed  connected and simply connected unless otherwise stated. 
It is a smooth manifold which is identified with $\bR^n$ via the exponential mapping and a choice of coordinate system. 
This leads to a corresponding Lebesgue measure on its Lie algebra $\fg$ and the Haar measure $dx$ on the group $G$,
hence $L^p(G)\cong L^p(\bR^n)$.
This also allows us \cite[p.16]{corwingreenleaf}
to define the spaces 
$$
\cD(G)\cong \cD(\bR^n)
\quad \mbox{and}\quad  
\cS(G) \cong \cS(\bR^n)
$$
 of test functions which are smooth and compactly supported or Schwartz, 
and the corresponding spaces of distributions 
$$
\cD'(G)\cong \cD'(\bR^n)
\quad \mbox{and}\quad 
\cS'(G)\cong \cS'(\bR^n).
$$
Note that this identification with $\bR^n$ does not usually extend to the convolution: the group convolution, i.e. the operation between  two functions on $G$ defined formally via 
$$
 (f_1*f_2)(x):=\int_G f_1(y) f_2(y^{-1}x) dy,
$$
 is   not commutative in general whereas it is a commutative operation for functions on  the abelian group $\bR^n$.
 
\subsubsection{Representations of $G$ and $L^1(G)$}
In this paper,  we always assume that the representations of the group $G$ 
are strongly continuous and acting on separable Hilbert spaces.
Unless otherwise stated, the representations of $G$ will also be assumed unitary.
For a representation $\pi$ of $G$, 
we denote by $\cH_\pi$ its Hilbert space, and 
we keep the same notation for the corresponding infinitesimal representation
which acts on the universal enveloping algebra $\fU(\fg)$ of the Lie algebra of the group.
It is characterised by its action on $\fg$:
\begin{equation}
\label{eq_def_pi(X)}
\pi(X)=\partial_{t=0}\pi(e^{tX}),
\quad  X\in \fg.
\end{equation}
The infinitesimal action acts on the space $\cH_\pi^\infty$
of smooth vectors, that is, the space of vectors $v\in \cH_\pi$ such that 
 the mapping $G\ni x\mapsto \pi(x)v\in \cH_\pi$ is smooth.

We will use the following equivalent notations for the group Fourier transform of a function  $f\in L^1(G)$
  at $\pi$ 
$$
 \pi(f) \equiv \widehat f(\pi) \equiv \cF_G(f)(\pi)=\int_G f(x) \pi(x)^*dx.
$$ 

\subsubsection{The Plancherel formula}
We denote by $\Gh$ the unitary dual of $G$,
that is, the unitary irreducible representations of $G$ modulo equivalence and identify a unitary irreducible representation 
with its class in $\Gh$. The set $\Gh$ is naturally equipped with a structure of standard Borel space.
The Plancherel measure is the unique positive Borel measure $\mu$ 
on $\Gh$ such that 
for any $f\in C_c(G)$, we have:
\begin{equation}
\label{eq_plancherel_formula}
\int_G |f(x)|^2 dx = \int_{\Gh} \|\cF_G(f)(\pi)\|_{HS(\cH_\pi)}^2 d\mu(\pi).
\end{equation}
Here $\|\cdot\|_{HS(\cH_\pi)}$ denotes the Hilbert-Schmidt norm on $\cH_\pi$.
This implies that the group Fourier transform extends unitarily from 
$L^1(G)\cap L^2(G)$ to $L^2(G)$ onto the Hilbert space
$$
L^2(\Gh):=\int_{\Gh} \cH_\pi \otimes\cH_\pi^* d\mu(\pi),
$$
which we identify with the space of $\mu$-square integrable fields $\sigma$ on $\Gh$ with Hilbert norm 
$$
\|\sigma\|_{L^2(\Gh)}=\sqrt{\int_{\Gh} \|\sigma(\pi)\|_{HS(\cH_\pi)} d\mu(\pi)}.
$$
Consequently \eqref{eq_plancherel_formula} holds for any $f\in L^2(G)$ and may be restated as
$$
\|f\|_{L^2(G)} = \|\widehat f\|_{L^2(\Gh)},
$$
this formula is called the Plancherel formula.
It is possible to give an expression for the Plancherel measure $\mu$, see \cite[Section 4.3]{corwingreenleaf}, although we will not need this in this paper.
We deduce the inversion formula: for any $\kappa\in \cS(G)$, 
\begin{equation}
\label{eq_FI}	
\forall x\in G\
\int_{\Gh} \tr(\pi(x)\cF_G\kappa(\pi) )d\mu(\pi) 
=\kappa(x).
\end{equation}

\subsubsection{The von Neumann algebra and $C^*$-algebra of $G$}
\label{subsubsec_vNG}

The von Neumann algebra of the group $G$ may be realised as the von Neumann algebra $\sL(L^2(G))^G$ of $L^2(G)$-bounded operators commuting with the left-translations on $G$.
As our group is nilpotent, the $C^*$-algebra of the group is then the closure of the space of right-convolution operators with convolution kernels in the Schwartz space. 

Dixmier's full Plancherel theorem  \cite[Ch. 18]{Dixmier} states that the von Neumann algebra of $G$ can also be  realised  as
the space
$L^\infty (\Gh)$
of measurable fields of operators that are bounded, that is, 
of measurable fields of operators
 $\sigma=\{\sigma(\pi)\in \sL(\cH_\pi): \pi\in \Gh\}$ such that
 $$
 \exists C>0\qquad \|\sigma(\pi)\|_{\sL(\cH_{\pi})} \leq C 
 \ \mbox{for} \ d\mu(\pi)\mbox{-almost all} \ \pi \in \Gh.
 $$
  The smallest of such constant $C>0$ is the norm $\|\sigma\|_{L^{\infty}(\Gh)}$ of $\sigma$  
  in $L^\infty(\Gh)$. 
 Similarly,  the $C^*$-algebra of the group $C^*(G)$ is then 
the closure of $\cF_G\cS(G)$ for the $L^\infty(\Gh)$-norm, 
and 
$L^\infty (\Gh)$ is the von Neumann algebra generated by the $C^*$-algebra of the group. 

The isomorphism between 
the von Neumann algebras $L^\infty(\Gh)$ and $\sL(L^2(G))^G$ are described as follows.
Clearly, the  Fourier multiplier $f\mapsto \cF_G^{-1} \sigma \widehat f$ of a symbol $\sigma\in L^\infty(\Gh)$ is in $\sL(L^2(G))^G$.
The converse is given by \cite[Ch. 18]{Dixmier}:
 if $T\in \sL(L^2(G))^G$, then 
there exists  a unique field $\widehat T \in L^\infty (\Gh)$ such that 
$T$ and $f\mapsto \cF_G^{-1} \widehat T \widehat f$ coincide; moreover, $\|\widehat T\|_{L^\infty(\Gh)}=\|T\|_{\sL(L^2(G))}$.
By the Schwartz kernel theorem, 
the operator $T$ admits a distributional convolution kernel  $\kappa\in \cS'(G)$.
We may also write $\widehat \kappa=\widehat T$ and call this field the group Fourier transform of $\kappa$ or of $T$. 
It extends the previous definition of the group Fourier transform on $L^1(G)$ and $L^2(G)$.

\subsection{Compact nil-manifolds}
A compact nil-manifold is the quotient $M=\Gamma\backslash G$ of 
a  nilpotent Lie group $G$  by a discrete co-compact subgroup $\Gamma$ of $G$.
A concrete example of discrete co-compact subgroup is the natural discrete subgroup of the Heisenberg group, as described in 
\cite[Example 5.4.1]{corwingreenleaf}. 
Abstract characterisations are discussed in  \cite[Section 5.1]{corwingreenleaf}.

An element of $M$ is a class 
$$
\dot x := \Gamma x
$$
 of an element $x$ in $G$. If the context allows it, we may identify this class with its representative $x$. 

The quotient $M$ is naturally equipped with the structure of a compact smooth manifold. 
Furthermore, fixing a Haar measure on the unimodular group $G$, 
$M$ inherits a measure $d\dot x$ which is invariant under the translations for each $g\in G$ given by
$$
\dot x  \longmapsto  \dot x g = \Gamma xg, 
\qquad 
M  \longrightarrow  M.
$$
Recall that the Haar measure $dx$ on $G$ is unique up to a constant and, once it is fixed, $d\dot x$ is the only $G$-invariant measure on $M$ satisfying 
for any  function $f:G\to \mathbb C$, for instance continuous with compact support,
\begin{equation}
\label{eq_dxddotx}
	\int_G f(x) dx = \int_M \sum_{\gamma\in \Gamma} f(\gamma x) \ d\dot x.
\end{equation}
We denote by $\vol (M) = \int_M 1 d\dot x$  the volume of $M$.

\subsection{$\Gamma$-periodic functions on $G$ and functions on $M$}
\label{subsec_periodicfcn}
Let $\Gamma$ be a discrete co-compact subgroup of a nilpotent Lie group $G$. 

We say that a function $f:G\rightarrow \mathbb C$  is  $\Gamma$-left-periodic or just  $\Gamma$-periodic  when we have 
$$
\forall x\in G,\;\;\forall \gamma\in \Gamma ,\;\; f(\gamma x)=f(x).
$$
This definition extends readily to measurable functions and to distributions.  

There is a natural one-to-one correspondence between the functions on $G$ which are $\Gamma$-periodic and the functions on $M$.
Indeed, for any map $F$ on $M$, 
the corresponding periodic function on $G$ is $F_G$ defined via
$$	 
F_G(x) := F(\dot x), \quad x\in G,
$$
while if $f$ is a $\Gamma$-periodic function on $G$, 
it defines a function $f_M$ on $M$ via
$$
f_M(\dot x) =f(x), \qquad x\in G.
$$
Naturally, $(F_G)_M=F$ and $(f_M)_G=f$.

We also  define, at least formally, the periodisation $\phi^\Gamma$ of a function $\phi(x)$ of the variable $x\in G$ by:
$$
\phi^\Gamma(x) = \sum_{\gamma \in \Gamma } \phi(\gamma x), \qquad x\in G.
$$

If $E$ is a space of functions or of distributions on $G$, then we denote by $E^\Gamma$ the space of elements in $E$ which are  $\Gamma$-periodic. 
Although 
$\cD(G)^\Gamma = \{0\} = \cS(G)^\Gamma,$
many other periodised functions or functional spaces have interesting descriptions on $M$ (see  \cite{nil-manifold} for comments and proofs):
\begin{proposition}
\label{prop_nil-manifold}
\begin{enumerate}
\item
 The periodisation of a Schwartz  function $\phi\in \cS(G)$ is a well-defined function  $\phi^\Gamma$ in $C^\infty(G)^\Gamma$.
Furthermore, the map $\phi \mapsto \phi^\Gamma$ yields a surjective morphism of topological vector spaces from
$\cS(G)$ onto $C^\infty(G)^\Gamma$
and from
$\cD(G)$ onto $C^\infty(G)^\Gamma$.
\item  For every $F\in \cD'(M)$, the tempered distribution  $F_G\in 
\cS'(G)$ is defined by
$$
\forall \phi\in \cS(G)\qquad 
\langle F_G,\phi\rangle  = \langle F , (\phi^\Gamma)_M\rangle.
$$
The map $F\mapsto F_G$ yields an isomorphism of topological vector spaces
from $\cD'(M)$ onto $\cS'(G)^\Gamma=\cD'(G)^\Gamma$.
\item 
For every $p\in [1,\infty]$,
the map $F\mapsto F_G$ is a topological vector space isomorphism of   (in fact, an isomorphism between Banach spaces) from $L^p(M)$ onto $L^p_{loc}(G)^\Gamma$ with inverse $f\mapsto f_M$.
\item 
Let $f\in \cS'(G)^\Gamma$ and $\kappa\in \cS(G)$.
Then $(\dot x, \dot y)  \mapsto \sum_{\gamma\in \Gamma} \kappa (y^{-1} \gamma x)$
is a smooth function on $M\times M$ and  $f*\kappa \in C^\infty(G)^\Gamma$.
Viewed as a function on $M$, 
$$
(f*\kappa)_M(\dot x) 
= \int_M f_M(\dot y) \sum_{\gamma\in \Gamma} \kappa (y^{-1} \gamma x) \  d\dot y.
$$
\item If $F\in L^p(M)$ with $p\in [1,+\infty]$ then 
$$
\|(F_G*\kappa)_M\|_{L^p(M)}
\leq \|\kappa\|_{L^1(G)} \|F\|_{L^p(M)}.
$$
\end{enumerate}
 \end{proposition}

A more precise norm for the $L^2$-boundedness follows 
from general considerations in harmonic analysis on homogeneous domains (see e.g. \cite[Lemma 2.5]{withSoren}):
\begin{lemma}
\label{lem_RegRep}
For any $\kappa\in L^1(G)$ and any $F_G\in L^2(M)$, 
$$
\|(F_G*\kappa)_M\|_{L^2(M)}
\leq \sup_{\pi\in \widehat \pi} \|\pi(\kappa)\|_{\sL(\cH_\pi)} \|F\|_{L^2(M)}.
$$
\end{lemma}

\subsection{Operators on $M$ and $G$}
\label{subsec_OpGM}
A mapping $T:\cS'(G)\to \cS'(G)$
or $\cD'(G) \to \cD'(G)$ is invariant under an element $g\in G$ when  
$$
\forall f\in \cS'(G) \ (\mbox{resp.} \, \cD'(G)), \qquad T(f(g \, \cdot)) = (Tf)(g \, \cdot).
$$ 
It is invariant under a subset of $G$ if it is invariant under every element of the subset.

Consider a linear continuous mapping $T:\cS'(G)\to \cS'(G)$
or $\cD'(G) \to \cD'(G)$ respectively
which is invariant under $\Gamma$. Then it  naturally induces
a linear continuous mapping
 $T_M$ on $M$ given via
$$
T_M F = (TF_G)_M, \qquad F\in \cD'(M).
$$
Consequently, 
if  $T$ coincides with   a smooth differential operator 	on $G$ that is invariant under $\Gamma$, then $T_M$ is a smooth differential operator on $M$.
For convolution operators $T$, we obtain
\begin{lemma}
\label{lem_IntKernel}
Let $\kappa\in \cS(G)$ be a given convolution kernel, and let us denote by $T$ the associated convolution operator:
$$
T (\phi) = \phi*\kappa, \qquad \phi\in \cS'(G).
$$
The operator $T$ is a linear continuous mapping $\cS'(G)\to \cS'(G)$. 
The corresponding operator $T_M$ maps  $\cD'(M)$ to $\cD'(M)$ continuously and linearly. Its  integral kernel is
$$
K\in C^\infty(M\times M), \qquad
K(\dot x,\dot y) = 
\sum_{\gamma\in \Gamma}  \kappa(y^{-1}\gamma  x).
$$
Consequently, the operator $T_M$ is Hilbert-Schmidt on $L^2(M)$ with Hilbert-Schmidt norm 
$$
\|T_M\|_{HS} = \|K\|_{L^2(M\times M)}.
$$
\end{lemma}
\begin{proof} This follows from 
the results in  Proposition \ref{prop_nil-manifold}. 
\end{proof}

\section{Preliminaries on Rockland operators  on $G$ and $M$}
\label{secR}

In this section, we recall the definition and known properties of Rockland operators. 
Our convention here is that they are homogeneous left-invariant operators on a nilpotent Lie group $G$ whose group Fourier transform is injective (see  Section \ref{subsubsec_posRdef}). This implies that the group $G$ is graded. We therefore start this section with describing the setting of graded groups. 

\subsection{Graded nilpotent Lie group}
\label{subsec_gradedG}

In the rest of the paper, 
we will be concerned with graded Lie groups.
References on this subject includes \cite{folland+stein_82} and 
\cite{R+F_monograph}.

\subsubsection{Definition}
\label{subsubsecDefgradedG}
A graded Lie group $G$  is a connected and simply connected 
Lie group 
whose Lie algebra $\fg$ 
admits an $\bN$-grading
$\fg= \oplus_{\ell=1}^\infty \fg_{\ell}$
where the $\fg_{\ell}$, $\ell=1,2,\ldots$, 
are vector subspaces of $\fg$,
almost all equal to $\{0\}$,
and satisfying 
$[\fg_{\ell},\fg_{\ell'}]\subset\fg_{\ell+\ell'}$
for any $\ell,\ell'\in \bN$.

This implies that the group $G$ is nilpotent.
Examples of such groups are the Heisenberg group
 and, more generally,
all stratified groups (which by definition correspond to the case $\fg_1$ generating the full Lie algebra $\fg$).
A Carnot group is by definition a stratified group
together with a scalar product on $\fg_1$.

\subsubsection{Dilations and homogeneity}
\label{subsubsec_dilations}
For any $r>0$, 
we define the  linear mapping $D_r:\fg\to \fg$ by
$D_r X=r^\ell X$ for every $X\in \fg_\ell$, $\ell\in \bN$.
We may also write $D_r X = r\cdot X$.
Then  the Lie algebra $\fg$ is endowed 
with the family of dilations  $\{D_r, r>0\}$
and becomes a homogeneous Lie algebra in the sense of 
\cite{folland+stein_82}.
The weights of the dilations are the integers $\ell\in \bN$ such that $\fg_\ell\not=\{0\}$, or in other words, the eigenvalues of $ (\ln r)^{-1}\ln D_r$ (counted without multiplicity). 
We denote by  
$$
\upsilon_1\leq \ldots \leq \upsilon_n
$$
the eigenvalues counted with multiplicities. The multiplicities correspond to the dimensions of $\fg_\ell\neq\{0\}$.  

We construct a basis $X_1,\ldots, X_n$  of $\fg$ adapted to the gradation,
by choosing a basis $\{X_1,\ldots X_{n_1}\}$ of $\fg_1$ (this basis is possibly reduced to $\emptyset$ if $\fg_1$ is trivial), then 
$\{X_{n_1+1},\ldots,  X_{n_1+n_2}\}$ a basis of $\fg_2$
(possibly $\emptyset$ as well as the others) etc.

 The associated group dilations are defined by
$$
D_r x = r\cdot x
:=(r^{\upsilon_1} x_{1},r^{\upsilon_2}x_{2},\ldots,r^{\upsilon_n}x_{n}),
\quad x=(x_{1},\ldots,x_n)\in G, \ r>0.
$$
In a canonical way,  this leads to the notions of homogeneity for functions, distributions and operators and we now give a few important examples. 

The Haar measure is $Q$-homogeneous where 
$$
Q:=\sum_{\ell\in \bN}\ell \dim \fg_\ell=\upsilon_1+\ldots+\upsilon_n,
$$
 is called the homogeneous dimension of $G$.

Identifying the element of $\fg$ with left invariant vector fields, 
each  $X_j$ is a $\upsilon_j$-homogeneous differential operator. More generally, the differential operator 
$$
X^{\alpha}=X_1^{\alpha_1}X_2^{\alpha_2}\cdots
X_{n}^{\alpha_n}, \quad \alpha\in \bN_0^n$$
is homogeneous with degree
$$
[\alpha]:=\alpha_1 \upsilon_1 + \cdots + \alpha_n \upsilon_n.
$$

The unitary dual $\Gh$ inherits  a dilation from the one on $G$ \cite[Section 2.2]{FFPisa}: we denote by $r\cdot \pi$ the element of $\Gh$ obtained from $\pi$ through dilations by $r$, that is, $r \cdot \pi(x) = \pi (r\cdot x)$, 
$r>0$ and $x\in G$.

\subsubsection{Homogeneous quasi-norms}
 
An important class of homogeneous maps are the homogeneous quasi-norms, 
that is, a $1$-homogeneous non-negative continuous map $G \ni x\mapsto \|x\|$ which is symmetric and definite in the sense that $\|x^{-1}\|=\|x\|$ and $\|x\|=0\Longleftrightarrow x=0$.
In fact, all the homogeneous quasi-norms are equivalent in the sense that if $\|\cdot\|_1$ and $\|\cdot\|_2$ are two of them, then 
$$
\exists C>0 \qquad \forall x\in G
\qquad C^{-1} \|x\|_1 \leq \|x\|_2 \leq C \|x\|_1.
$$
Examples may be constructed easily, such as  
$$
\|x\| = (\sum_{j=1}^n |x_j|^{N/\upsilon_j})^{-N} \ \mbox{for any}\ N\in \bN,
$$
 with the convention above and the weights $\upsilon_j$ introduced in Section \ref{subsubsec_dilations}. 
 In the stratified case, it is also possible to construct a homogeneous quasi-norm which is also a norm. 
 
\subsubsection{Approximation of the identity on $G$ and $M$}

Using the dilations, it is easy to adapt the construction of approximations of the identity from the Euclidean setting to the context of  graded groups \cite[Section 3.1.10]{R+F_monograph}:
\begin{proposition}
\label{prop_app_id}
	Let $\kappa\in \cS(G)$ with $\int_G \kappa(y)dy=1$. 
	Set $\kappa_t (y) = t^{-Q}\kappa(t^{-1}\cdot y)$. 
For any $p\in [1,\infty)$, 
we have for any $f\in L^p(G)$:
$$
\lim_{t\to 0} \|f* \kappa^{(t)} - f\|_{L^p(G)}
=0=
\lim_{t\to 0} \| \kappa^{(t)} *f - f\|_{L^p(G)},
$$
\end{proposition}
%
%
%
\subsection{Positive Rockland operators on $G$}
\label{subsec_cR}
Let us briefly review the definition and main properties of positive Rockland operators. 
References on this subject includes \cite{folland+stein_82} and 
\cite{R+F_monograph}.

\subsubsection{Definitions}
\label{subsubsec_posRdef}

A \emph{Rockland operator}
 $\cR_G$ on $G$ is 
a left-invariant differential operator 
which is homogeneous of positive degree and satisfies the Rockland condition, that is, 
for each unitary irreducible representation $\pi$ on $G$,
except for the trivial representation, 
the operator $\pi(\cR_G)$ is injective on the space  $\cH_\pi^\infty$ of smooth vectors of the infinitesimal representation.

\begin{remark}
Some authors may not assume the homogeneity of the Rockland operator, but as a convention here, we assume that Rockland operators are homogeneous. 
This convention does not change the analysis since authors who do not assume homogeneity study the principal part of the operators, that is, what we call Rockland operator. 
\end{remark}

Recall
that Rockland operators $\cR_G$ are hypoelliptic.
In fact, they are equivalently characterised as the left-invariant homogeneous differential operators that are hypoelliptic. 
If this is the case, then $\cR_G + \sum_{[\alpha]< \nu}c_\alpha X^\alpha$, where $c_\alpha\in \bC$ and $\nu$ is the homogeneous degree of $\cR$, is hypoelliptic.

A Rockland operator is \emph{positive} when 
$$
\forall f \in \cS(G),\qquad
\int_G \cR_G f(x) \ \overline{f(x)} dx\geq 0.
$$

Any sub-Laplacian with the sign convention $-(X_1^2+\ldots+X_{n_1}^2)$ of a stratified Lie group  is a positive Rockland operator; here $X_1,\ldots, X_{n_1}$ form a basis of the first stratum $\fg_1$.
The reader familiar with the Carnot group setting may 
 view positive Rockland operators as generalisations of the natural sub-Laplacians.
Positive Rockland operators are easily constructed on any graded Lie group, see \cite[Corollary 4.1.10]{R+F_monograph}.

A positive Rockland operator is essentially self-adjoint on $L^2(G)$ and we keep the same notation for their self-adjoint extension.
Its spectrum is $\sp(\cR_G)$ included in $[0,+\infty)$ and the point 0 may be neglected in its spectrum \cite[Lemma 2.12]{FFPisa}.

For each unitary irreducible representation $\pi$ of $ G$, 
the operator 
$\pi(\cR_G)$ is  essentially self-adjoint on $\cH^\infty _\pi$,
and we keep the same notation for this self-adjoint extension.
Its spectrum $\sp(\pi(\cR_G))$ is a discrete subset of $(0,\infty)$ if $\pi\not =1_{\Gh}$ is not trivial, while $\pi(\cR_G)=0$ if $\pi=1_{\Gh}$ is the trivial representation, see e.g. \cite{nil-manifold}.

Let us denote by $E$ and $E_\pi$ the spectral measures 
$$
\mbox{of} \ \cR_G = \int_\bR \lambda dE_\lambda,
\quad\mbox{and of}\ 
\pi(\cR_G) = \int_\bR \lambda dE_\pi(\lambda),  \quad \pi\in \Gh.
$$
Then $\widehat E(\pi)=E_\pi$ in the sense that 
for any interval $I\subset \bR$, 
the group Fourier transform 
$\widehat {E(I)}$ of the projection $E(I)\in \sL(L^2(G))^G$ coincides with the field of operator
$$
\widehat E(I):=\{\pi(E(I)), \pi\in \Gh\}
$$
which  is in $L^\infty(\Gh)$. 
	It satisfies:
	$$
	\widehat E(I)^2=\widehat E(I)
	\quad\mbox{and when not trivial}\quad
	\|\widehat E(I)\|_{L^\infty(\Gh)} =1.
	$$
	We will use the notation for any $\lambda\in \bR$
	\begin{equation}
		\label{eq_Elambda}
		E_\lambda = E(\{\lambda\})
	\qquad\mbox{and}\qquad  
	\widehat E(\{\lambda\}) = \widehat E_\lambda  .
	\end{equation}

\subsubsection{Spectral multipliers in $\cR_G$ and  in $\widehat {\cR_G}$}
\label{subsubsec_psi(R)}

If $\psi:\bR^+\to \bC$ is a measurable function,
the spectral multiplier $\psi(\cR_G) = \int_\bR \psi(\lambda) dE_\lambda$ is well defined as a possibly unbounded operator on $L^2(G)$.
If the domain of $\psi(\cR_G)$  contains $\cS(G)$ 
and defines a continuous map $\cS(G)\to \cS'(G)$, then 
it is invariant under right-translation and, by the Schwartz kernel theorem, admits a right-convolution kernel $\psi(\cR_G)\delta_0 \in \cS'(G)$ which satisfies the following homogeneity property:
\begin{equation}
\label{eq_homogeneitypsiR}	
 \psi(r^\nu \cR_G) \delta_0 (x) =r^{-Q} \psi(\cR_G)\delta_0(r^{-1}\cdot x),
 \quad x\in G.
\end{equation}
Furthermore, for each unitary irreducible representation $\pi$ of $G$,  
 the domain of the operator
$ \psi(\pi(\cR_G))= \int_\bR \psi(\lambda) dE_\pi (\lambda) $
contains $\cH_\pi^\infty$ and we have
$$
\widehat  {\psi(\cR_G)}(\pi)=
 \psi (\pi(\cR_G)) .
$$

\smallskip

The following statement is the famous result due to Hulanicki \cite{hulanicki}:
\begin{theorem}[Hulanicki's theorem]
\label{thm_hula}
	Let $\cR_G$ be a positive Rockland operator on $G$.
	If $\psi\in \cS(\bR)$ then $\psi(\cR_G)\delta_0 \in \cS(G)$.
\end{theorem}

For instance,  the heat kernels 
$$
p_t:=e^{-t\cR_G}\delta_0, \quad t>0,
$$
 are Schwartz - although this property is in fact used in the proof of Hulanicki's Theorem. 

\medskip

The following result describes the map 
  $\psi\mapsto \psi(\cR_G)\delta_0$.
This was mainly obtained by Christ for sub-Laplacians on stratified groups \cite[Proposition 3]{Christ91} and readily extended   to positive Rockland operators in \cite{nil-manifold}, see also \cite{Martini}   and the references to Hulanicki's works therein.

\begin{theorem}
\label{thm_christ}[Christ, Hulancki]
	Let $\cR_G$ be a positive Rockland operator of homogeneous degree $\nu$ on $G$. 
	If the measurable function  $\psi:\bR^+\to \bC$ is in  $L^2(\bR^+,  \lambda^{Q/\nu} d\lambda/\lambda)$, 
	then $\psi(\cR_G)$  defines a continuous map $\cS(G)\to \cS'(G)$ whose convolution kernel  $\psi(\cR_G)\delta_0$ is in $L^2(G)$. Moreover,   we have 
$$
\|\psi(\cR_G)\delta_0\|_{L^2(G)}^2
 	=c_0\int_0^\infty |\psi (\lambda)|^2  \lambda^{\frac Q \nu} \frac{d\lambda}{\lambda},
 $$
 where $c_0 = c_0(\cR_G)$ is a positive constant of $\cR_G$ and $G$.  For any $\psi\in \cS(\bR)$, we also have
$$
\psi(\cR_G)\delta_0(0)
 	=c_0\int_0^\infty \psi (\lambda)  \lambda^{\frac Q \nu} \frac{d\lambda}{\lambda}.
 $$

\end{theorem}

\begin{remark}
\label{rem_thm_christ}
By plugging  the function $\psi(\lambda) =e^{-\lambda}$ in the second formula of Theorem \ref{thm_christ}, 
we obtain the following expression for the constant in the statement in terms of the heat kernel $p_t$ of $\cR_G$:
\begin{equation}
\label{eq_c0}
	c_0 = c_0(\cR_G) = \frac{p_1(0)}{\Gamma(Q/\nu)},
\quad\mbox{where}\quad 
\Gamma(Q/\nu) =\int_0^\infty e^{-\lambda}  \lambda^{Q/\nu} \frac{d\lambda}{\lambda}.
\end{equation} 
See  \cite[Section 3.3.3]{nil-manifold} for  further comments. 
\end{remark}

\subsection{Positive Rockland operators on $M$}
\label{subsec_propRM}

This section is devoted to the general properties of positive Rockland operators on nil-manifolds. 
They are well known \cite{nil-manifold,withSoren} 
in the context of nil-manifolds, and have been obtained  more generally on filtered manifolds  
\cite{Dave+Haller2020,Dave+Haller2022}, see also 
 \cite{CdvHT,CdvHT2}  on sub-Riemannian manifolds for sub-Laplacians.

\begin{proposition}
\label{prop_RM}
Let $\cR$ be a positive Rockland operator on $G$.
The operator $\cR_M$  it induces on $M$ is a smooth differential operator which is positive and essentially self-adjoint on $C^\infty(M)\subset L^2(M)$.

 \begin{enumerate}
 \item
  Let $\psi\in \cS(\bR)$.
  \begin{itemize}
\item[(i)] 	The operator $\psi(\cR_M)$ defined as a bounded spectral multiplier on $L^2(M)$ coincides with the operator 
$$
\phi \longmapsto (\psi (\cR) \phi_G	)_M= (\phi_G * \kappa
)_M
$$
where  $\kappa:=\psi( \cR)\delta_0\in \cS(G)$.
 The integral kernel of $\psi(\cR_M)$ is a smooth function on $M\times M$ given by 
  $$
K(\dot x,\dot y) = 
\sum_{\gamma\in \Gamma}  \kappa(y^{-1}\gamma  x).
$$
\item[(ii)]  
  For every $\eps\in (0,1]$,  the integral kernel $K_\eps$ of $
\psi(\eps^\nu \cR_M)$ satisfies 
$$
K_\eps (\dot x,\dot x)= \eps^{-Q} \kappa(0) +O(\eps)^\infty, 
$$
meaning that $K_\eps (\dot x,\dot x)= \eps^{-Q} \kappa(0) +O(\eps^N),$ for every $N\in \bN$.
Here $\nu$ is the degree of homogeneity of $\cR$, and $\kappa(0)$ is the value at $x=0$ of the convolution kernel $\kappa= \psi(\cR)\delta_0$ also given by Theorem \ref{thm_christ} as
 $$
\kappa(0)=c_0\int_0^\infty \psi (\lambda)  \lambda^{\frac Q \nu} \frac{d\lambda}{\lambda}.
$$
The trace and Hilbert-Schmidt norm have the following asymptotics in $\eps\to0$:
\begin{align*}
	\tr \left(\psi(\eps^\nu \cR_M)\right) 
& = \eps^{-Q} 
\vol (M)\kappa(0) + O(\eps^\infty)\\
\|\psi(\eps^\nu \cR_M)\|^2_{HS} 
&= \eps^{-Q} 
\vol (M)\, c_0\int_0^\infty |\psi (\lambda)|^2  \lambda^{\frac Q \nu} \frac{d\lambda}{\lambda} + O(\eps^\infty).
\end{align*}
\end{itemize}

\item The spectrum $\sp(\cR_M)$ of 	$\cR_M$ is a discrete and unbounded subset of $[0,+\infty)$.
Each eigenspace of $\cR_M$ has finite dimension.
The resolvent operators $(\cR_M -z)^{-1}$, $z\in \bC \setminus {\rm Sp}(\cR_M)$,  are  compact on $L^2(M)$.
 The constant functions on $M$ form the 0-eigenspace of $\cR_M$.
 The eigenfunctions of $\cR_M$ are smooth on $M$.
  \end{enumerate}
\end{proposition}

\begin{proof}
For Part (1)(i) and Part (2), see \cite[Section 3.4]{nil-manifold}.
See 
\cite[Section 4.1]{nil-manifold} for Part (1)(i).
\end{proof}

The asymptotics for the heat kernel of $\cR_M$ given above together with 
the homogeneous nature of $M=\Gamma \backslash G$ and
Karamata's tauberian theorem imply the following Weyl laws:

\begin{corollary}
\label{cor_prop_RM}
	We continue with the setting above. 
Consider an orthonormal basis $(\varphi_j)_{j\in \bN_0}$  of the Hilbert space $L^2(M)$ consisting\ of eigenfunctions of $\cR_M$:
$$
\cR_M \varphi_j = \mu_j \varphi_j,
\quad\mbox{with}\quad 
\mu_0 < \mu_1 \leq \mu_2 \leq  \ldots \leq \mu_j   \longrightarrow  +\infty
\quad\mbox{as}\quad j\to +\infty.
$$
We denote by $\nu$ the homogeneous degree of $\cR$ and 
its spectral counting function by
$$
N(\Lambda):=\left|\left\{ j\in\bN_0,\;\;\mu_j\leq \Lambda\right\}\right| .
$$
The Weyl law for $\cR_M$ is given by 
$$
N(\Lambda)\sim c \Lambda^{Q/\nu },
\qquad c:= \vol (M) \frac{ \nu}{Q} c_0 ,
$$
Moreover, 
for any continuous function $f:M\to \bC$,
 we have the following mean convergence:
$$
\lim_{\Lambda\to \infty} \frac 1{N(\Lambda)} 
\sum_{j: \mu_j \leq \Lambda }  \int_M f(\dot x) |\varphi_j(\dot x)|^2 d\dot x =  
\int_M f(\dot x) \frac{d\dot x}{\vol(M)}.
$$
\end{corollary}

\begin{proof}[Proof of Corollary \ref{cor_prop_RM}]
By density, 
it suffices to show the statement for $f\in \cD(M)$ smooth.
The Laplace transform of 
the measure $\mu$ given by 
	$$
	\mu[0,\Lambda] = \sum_{j: \mu_j \leq \Lambda } \int_M f(\dot x) |\varphi_j(\dot x)|^2 d\dot x
	= \tr (f(\dot x)  1_{[0,\Lambda]}(\cR_M)), \qquad \Lambda\geq 0,
		$$
 satisfies by the kernel estimate in Proposition \ref{prop_RM} (1) (ii):
	$$
	\int_0^\infty e^{-t\lambda} d\mu(\lambda) =
	\tr (f(\dot x)  e^{-t \cR}) = t^{-\frac Q\nu }  p_1(0) \int_M f(\dot x)d\dot x +O(t)^\infty.
	$$
By Karamata's tauberian theorem
	\cite[Theorem 10.3]{Simon}, we obtain:
	$$
	\lim_{\Lambda\to +\infty}\Lambda^{- Q/\nu}  \mu[0,\Lambda] 
	=\frac{p_1(0) \int_M f(\dot x)d\dot x}{\Gamma(1+\frac Q\nu)} = \frac{ \nu}{Q} c_0  \int_M f(\dot x)d\dot x,
	$$
	by \eqref{eq_c0}. This implies the Weyl law when  $f=1$, and then the full statement.  
\end{proof}

\section{Semi-classical calculus on graded compact nil-manifold}
\label{sec_sccM}

In this section, we discuss the semi-classical calculus on graded compact nil-manifolds and for smoothing symbols.  

\subsection{Semi-classical pseudodifferential operators} 
\label{subsec_sclpdo}

The semi-classical pseudodifferential calculus 
in the context of groups of Heisenberg type was presented in \cite{FFPisa,FFJST}, but in fact extends readily to any graded group $G$. 
Here, we show how to define it on the quotient manifold $M$. 

We denote by $\cA_0={\cA}_0(M\times \Gh)$ the class of symbols, that is  of fields of operators defined on $M\times \Gh$ 
$$
\sigma(\dot x,\pi)\in \sL({\cH}_\pi),\;\;(\dot x,\pi)\in M\times \Gh,
$$
that are of the form 
$$
\sigma(\dot x,\pi) = \cF_G \kappa_{\dot x} (\pi),
$$
where $\dot x\mapsto \kappa_{\dot x}$ is a smooth  map from $M$ to $\cS(G)$. 
The group Fourier transform yields a bijection $(\dot x\mapsto \kappa_{\dot x}) \mapsto (\dot x \mapsto  \sigma(\dot x,\cdot) = \cF_G(\kappa_{\dot x}))$ from $C^\infty(M;\cS(G))$ onto $\cA_0$.  We equip $\cA_0$ with the Fr\'echet topology so that this mapping is an isomorphism of topological vector spaces. 

We observe that $\cA_0$ is an algebra for the usual composition of symbols. Furthermore, it is also equipped with the involution $\sigma \mapsto \sigma^*$, where $\sigma^* = \{\sigma(\dot x,\pi)^* , (\dot x,\pi)\in M\times \Gh\}$.

\medskip

 Note that by the Fourier inversion formula \eqref{eq_FI}, we have
 $$
 \kappa_{\dot x}(z) 
 =\int_{\Gh} \tr(\pi(z)\, \sigma(\dot x,\pi) )d\mu(\pi) 
 =\int _{\Gh} \tr(\pi(z)\, \sigma_G( x,\pi) )d\mu(\pi).
 $$

For any $\sigma\in \cA_0$, 
we  define the operator $\Op_G (\sigma)$ at $F\in \cS'(G)$  via
$$
\Op_G (\sigma) F(x) := F*\kappa_{\dot x}(x), 
\quad x\in G.
$$ 
This makes sense since, for each $x\in G$, the convolution of the tempered distribution $F$ with the Schwartz function  $\kappa_{\dot x}$
yields a smooth function $F*\kappa_{\dot x}$ on $G$.
Because of the Fourier inversion formula \eqref{eq_FI}, it may be written formally as 
$$
\Op_G(\sigma)F(x)= 
\int_{G\times \Gh} 
\tr 
( \pi(y^{-1} x) 
\sigma_G 
(x, \pi) ) 
F(y) dy d\mu(\pi).
$$

If $F$ is periodic, then $\Op_G(\sigma) F$ is also periodic with 
$\Op_G(\sigma) F \in C^\infty(G)^\Gamma$
and
we can view $F$ and $\Op_G(\sigma) F$ as functions on $M$, 
see Section \ref{subsec_periodicfcn}. 
In other words, 
we set for any $f\in \cD'(M)$ and $\dot x\in M$:
$$
\Op(\sigma) f (\dot x) := \Op_G(\sigma) f_G (x) =
(f_G *\kappa_{\dot x})_M(\dot x)
= \int_M f(\dot y) \sum_{\gamma\in \Gamma} \kappa_{\dot x} (y^{-1} \gamma x)\,  d\dot y
,
$$
and this defines the function $\Op(\sigma) f \in \cD(M)$.
We say that $\kappa_{\dot x}$  is the \emph{kernel associated} with the symbol $\sigma$ or  $\Op(\sigma)$. It satisfies the following properties:
\begin{lemma}
\label{lem_IntKernelOpsigma}
Let $\sigma\in \cA_0$ and let 	$\kappa_{\dot x}$  be its associated  kernel.
Then $\Op(\sigma)$ maps $\cD'(M)$ to $\cD(M)$ continuously, and its Schwartz integral is the smooth function $K$ on $M\times M$ given by
$$
K(\dot x,\dot y) = 
\sum_{\gamma\in \Gamma}  \kappa_{\dot  x}(y^{-1}\gamma  x).
$$
Consequently, the operator $\Op(\sigma)$ is Hilbert-Schmidt on $L^2(M)$ with Hilbert-Schmidt norm 
$$
\|\Op(\sigma)\|_{HS} = \|K\|_{L^2(M\times M)}.
$$
\end{lemma}
\begin{proof}
The results follow from the properties in  Sections \ref{subsec_periodicfcn} and \ref{subsec_OpGM}.
\end{proof}

Let $\eps\in (0,1]$ be a small parameter.
For every symbol $\sigma\in{\cA}_0$, we consider the symbol
\begin{equation}
\label{eq_sigmaeps}
\sigma^{(\eps)}:=
\{\sigma(\dot x,\eps\cdot  \pi) : (\dot x, \pi)\in M\times \Gh\},
\end{equation}
whose associated kernel is then 
\begin{equation}
\label{eq_kappaeps}	
\kappa^{(\eps)}_{\dot x}(z):= \eps^{-Q} \kappa_{\dot x}(\eps^{-1}\cdot z), \quad z\in G,
\end{equation}
 if $\kappa_{\dot x}=\kappa^{(1)}_{\dot x} $  is the kernel associated with the symbol $\sigma=\sigma^{(1)}$.
The semi-classical pseudo-differential calculus is then  defined via
$$
\Op^\eps (\sigma) :=  \Op (\sigma^{(\eps)})
\quad\mbox{and}\quad
\Op_G^\eps (\sigma) :=  \Op_G (\sigma^{(\eps)}).
$$

An interesting example is given by the spectral multiplier in a positive Rockland operator $\cR$ on $G$. 
For any  $\psi\in \cS(\bR)$,
the operator $\psi(\cR_M)$ defined spectrally as a bounded spectral multiplier on $L^2(M)$ coincides with the operator $\Op (\sigma)$ on $C^\infty(M)$
  with symbol $\sigma(\pi):=\psi( \widehat{\cR}(\pi))$ in $\cA_0$ independent of $\dot x\in M$:
  $$
\psi(\cR_M) = \Op( \psi (\widehat \cR)).
$$ 	
  The associated kernel   is $\kappa:=\psi( \cR(\pi))\delta_0\in \cS(G)$.  The integral kernel of $\psi(\cR_M)$ is a smooth function on $M\times M$ given by 
  $$
K(\dot x,\dot y) = 
\sum_{\gamma\in \Gamma}  \kappa(y^{-1}\gamma  x).
$$

For every $\eps\in (0,1]$,  we have $
\psi(\eps^\nu \cR_M) = \Op^{(\eps)} (\sigma)$ on $C^\infty(M)$,  
where $\nu$ is the degree of homogeneity of $\cR$.

In the rest of this section, we give the general properties of the semi-classical calculus, starting with the boundedness on $L^2$. 

\subsection{Boundedness in $L^2(M)$}

First, let us introduce the following norm on $\cA_0$:
$$
\|\sigma \|_{\cA_0} := \int_{G} \sup_{\dot x\in M}  |\kappa_{\dot x}(y)|dy, 
$$
where $\kappa_{\dot x}$ is the kernel associated with $\sigma \in \cA_0$.
This is a continuous seminorm on $\cA_0$.
Later on, we will use another continuous seminorm on $\cA_0$, which is given by 
\begin{equation}
\label{eq_Linftysigma}
\|\sigma\|_{L^\infty(M\times \Gh)}
:=
\sup_{(\dot x,\pi)\in M\times \Gh} \|\sigma(\dot x,\pi)\|_{\sL(\cH_\pi)}.	
\end{equation}
Here the supremum is the essential supremum with respect to the Plancherel measure. 
Note that since $\|\pi(f) \|_{\sL(\cH_\pi)} \leq \|f\|_{L^1(G)}$ for any $f\in L^1(G)$, we have
\begin{equation}
	\label{eq_normsLinftyA0}
\|\sigma\|_{L^\infty(M\times \Gh)} \leq \sup_{\dot x\in M} \|\kappa_{\dot x}\|_{L^1(G)} \leq \|\sigma \|_{\cA_0}.
\end{equation}

The main property of the semi-classical calculus regarding $L^2$-boundedness is the following:
\begin{proposition}
\label{prop_bddL2}
For every $\eps>0$ and $\sigma\in \cA_0$, 
$$
\| \Op^\eps (\sigma)\|_{\mathcal L(L^2(M))} \leq  \|\sigma^{(\eps)}\|_{\cA_0}
 =\|\sigma\|_{\cA_0}
$$
where $\sigma^{(\eps)}$ is given in \eqref{eq_sigmaeps}. \end{proposition}

\begin{proof}
The equality in the statement follows from 
a simple change of variable $y=\eps^{-1}\cdot z$ in 
$$
\|\sigma^{(\eps)}\|_{\cA_0}
=\|\sup_{\dot x_1\in M}|\kappa^{(\eps)}_{\dot x_1}|\|_{L^1(G)}
=\int_G \sup_{\dot x_1\in M}|\kappa_{\dot x_1}|(\eps^{-1}\cdot z) \ \eps^{-Q} dz
=\int_G \sup_{\dot x_1\in M}|\kappa_{\dot x_1}(y)|  dy = \|\sigma\|_{\cA_0}.
$$
Hence it suffices to show the case of $\eps=1$. We observe that 
we have for any $f\in \cD(M)$,
$$
|\Op(\sigma) f (\dot x)|
=\big| \int_M f(\dot y) \sum_{\gamma\in \Gamma} \kappa_{\dot x} (y^{-1} \gamma x)\,  d\dot y\big|
\leq  \int_M |f(\dot y)| \sum_{\gamma\in \Gamma} \sup_{\dot x_1\in M}|\kappa_{\dot x_1} (y^{-1} \gamma x)|\  d\dot y,
$$
consequently using \eqref{eq_dxddotx}
\begin{align*}
\|\Op(\sigma) f \|_{L^\infty(M)} 
&\leq \|f\|_{L^\infty(M)}
 \int_M \sum_{\gamma\in \Gamma} \sup_{\dot x_1\in M}|\kappa_{\dot x_1} (y^{-1} \gamma x)|\  d\dot y
 = \|f\|_{L^\infty(M)}\int_{G} \sup_{\dot x\in M}  |\kappa_{\dot x}(y)|dy,
 \\
 \|\Op(\sigma) f \|_{L^1(M)}& \leq 
\int_M |f(\dot y)| \int_{z\in G}\sup_{\dot x_1\in M}|\kappa_{\dot x_1} (y^{-1} z)| dz \  d\dot y
=
\|f\|_{L^1(M)}\int_{G} \sup_{\dot x\in M}  |\kappa_{\dot x}(z')|dz'.	
\end{align*}
In other words, the linear map $\Op(\sigma)$ extends continuously as an operator $L^p(M)\to L^p(M)$ for $p=1,\infty$ with norm $\leq \|\sigma\|_{\cA_0} $.
By interpolation (Riesz-Thorin theorem), we obtain the first inequality in the statement. 
\end{proof}

\subsection{Singularity of the operators as $\eps \to 0$.}
The following lemma is similar to Proposition 3.4 in \cite{FFchina}
and shows that the singularities of the integral kernels of the operators $\Op^{(\eps)}(\sigma)$ concentrate on the diagonal as $\eps\to 0$.
It may also justify for many semi-classical properties that
 the kernel associated with a symbol may be assumed to be compactly supported in the variable of the group:

\begin{lemma}
\label{lem_SvsCc}
	Let $\eta\in \cD(G) $ be identically equal to $1$ close to $0$. 
	Let $\sigma\in \cA_0$ and let $\kappa_{\dot x}(z)$ denote its associated kernel.
	For every $\eps>0$, the symbol $\sigma_\eps$ defined via
	$$
	\sigma_\eps (\dot x,\pi)   = \cF_G \left( \kappa_{\dot x} \eta (\eps \, \cdot)\right),
	$$ 
	that is, the symbol with associated kernel $\kappa_{\dot x}(z) \eta(\eps\cdot  z)$, is in $\cA_0$.
	Furthermore,
for all $N\in\bN$, there exists a constant $C=C_{N,\sigma,\eta}>0$ such that 
$$
\forall \eps\in (0,1]\qquad
\left\|  \sigma_\eps-\sigma\right\|_{\cA_0}
\leq C {\eps}^{N}.
$$
\end{lemma}

\begin{proof}
As the function $\eta$ is identically~$1$ close to $z=0$, for all $N\in\bN$, there exists a bounded continuous function $\theta$, identically 0 near 0,
such that 
$$
\forall y\in G,\;\; 1-\eta(y)=\theta(y) \| y\|^N,
$$
where $\|\cdot\|$ is a fixed homogeneous quasi-norm on $G$ (see Section \ref{subsubsec_dilations}). 
This notation implies
$$
\kappa_{\dot x}(z) - \kappa_{\dot x}(z) \eta(\eps\cdot  z)
= \kappa_{\dot x}(z) \theta (\eps \cdot z)  \|\eps \cdot z\|^N.
$$
As $\|\eps \cdot z\|=\eps \|z\|$, we obtain
$$
\left\|  \sigma_\eps - \sigma\right\|_{\cA_0}
=\int_G \sup_{\dot x}|\kappa_{\dot x}(z) - \kappa_{\dot x}(z) \eta(\eps\cdot  z)|dz
\leq \eps^N \|\theta\|_{L^\infty} \int_G \sup_{\dot x\in M} | \kappa_{\dot x}(z)| \|z\|^N dz.
$$
This last integral is finite and this concludes the proof.
\end{proof}

\subsection{Symbolic calculus}
\label{subsec_Symbolic_Calculus}

In order to present 
 the symbolic properties of the semi-classical calculus,
 we first need to introduce the notions of difference operators.
They aim at replacing the derivatives with respect to the Fourier variable in the Euclidean case.

If $q$ is a smooth function on $G$ with polynomial growth, 
we define the associated difference operator $\Delta_q$ via
$$
\Delta_q \widehat f (\pi) = \cF_G(q f) (\pi),
\quad \pi\in \Gh,
$$
for any function $f\in \cS(G)$ and even any tempered distribution $f\in \cS'(G)$.
On the commutative group $(\bR^n,+)$ if $q=x_j$ then $\Delta_q$ is the derivative $\partial_{\xi_j}$ on $j$th component of the Fourier variable (up to a constant depending on the convention about the Euclidean Fourier transform). 
In particular, considering the basis $(X_j)$ constructed in Section \ref{subsubsec_dilations}, 
we consider its dual  \cite[Proposition 5.2.3]{R+F_monograph} formed by the polynomials $q_\alpha$ such that $X^\beta q_\alpha=\delta_{\alpha,\beta}$ for all $\alpha,\beta\in \bN_0^n$.
We then define the difference operator associated with $q_\alpha(\cdot^{-1}): y \mapsto q_\alpha(y^{-1})$:
$$
\Delta^\alpha:=\Delta_{q_\alpha(\cdot^{-1})}.
$$

We obtain readily the following symbolic properties of the  calculus:
\begin{proposition}
\label{prop_symbolic_calculus}
\begin{enumerate}
\item 	If  $\sigma_1,\sigma_2\in \cA_0$, 
then 
$$
\Op^{(\eps)}(\sigma_1)\Op^{(\eps)}(\sigma_2)=
 \sum_{[\alpha]\leq N } \eps^{[\alpha]}\Op^{(\eps)}( \Delta^\alpha \sigma_1 \, X^\alpha_M \sigma_2) 
 \ + \ O( \eps)^{N+1},
$$
for any $N\in \bN_0$ in the $\cL(L^2(M))$ sense, that is, 
$$
\forall \eps\in (0,1]\qquad
\|\Op^{(\eps)}(\sigma_1)\Op^{(\eps)}(\sigma_2)
- \sum_{[\alpha]\leq N } \eps^{[\alpha]}\Op^{(\eps)}( \Delta^\alpha \sigma_1 \, X^\alpha_M \sigma_2) \| _{\cL(L^2(M))} \lesssim_{N,\sigma_1,\sigma_2} \eps^{N+1}.
$$
\item 	If  $\sigma\in \cA_0$, 
 then 
 $$
\Op^{(\eps)}(\sigma)^*
=
\sum_{[\alpha]\leq N } \eps^{[\alpha]}
\Op^{(\eps)}( \Delta^\alpha  X^\alpha \sigma^*)
 \ + \ O(\eps)^{N+1},
$$
for any $N\in \bN_0$ in the $\cL(L^2(M))$ sense, that is, 
 $$
 \forall \eps\in (0,1]\qquad
\|\Op^{(\eps)}(\sigma)^*
-
\sum_{[\alpha]\leq N } \eps^{[\alpha]}
\Op^{(\eps)}( \Delta^\alpha  X^\alpha \sigma^*)
  \| _{\cL(L^2(M))} \lesssim_{N,\sigma} \eps^{N+1}.
$$\end{enumerate}
\end{proposition}

The proof of Proposition \ref{prop_symbolic_calculus} is a straightforward adaptation of similar arguments
on  calculi   on $G$ and $M$ with symbol classes, see e.g. \cite{withSoren} or  \cite[Section 5.5]{R+F_monograph} (see also \cite{FFchina,FFJST,FFFGeomInv,FFFMilan,Fermanian+Letrouit}):
 
\begin{proof} We follow closely \cite{withSoren}, more precisely the proof of Theorem 5.1 therein given in its Appendix A.
We observe that   $\cA_0$ coincides with the smoothing symbols $S^{-\infty}(M\times \Gh)$ by \cite[Corollary 4.19]{withSoren}. 
Let $\sigma_1,\sigma_2 \in \cA_0$.
By \cite[Theorem 4.11]{withSoren},
$$
\Op^{(\eps)}(\sigma_1)\Op^{(\eps)}(\sigma_2)
=
\Op^{(\eps)}(\sigma)
$$
with the symbol $\sigma$
depending on $\eps\in (0,1]$ given by 
its convolution kernel  
$$
\kappa_{\dot x}(y)= \eps^Q \int_G \kappa^{(\eps)}_{\sigma_2, \Gamma xz^{-1}}( (\eps \cdot y)z^{-1})\ \kappa^{(\eps)}_{1,x}(z) dz 
=\int_G \kappa_{\sigma_2,\Gamma x(\eps \cdot z)^{-1}}( y z^{-1}) \ \kappa_{\sigma_1,\dot x}(z) dz .
$$
By the Taylor estimates due to Folland 
and Stein \cite{folland+stein_82} (also stated in \cite[Appendix A.1]{withSoren}), having fixed a quasi-norm $\|\cdot\|$, we have for any $N\in \bN$
$$
| \kappa_{\sigma_2,\Gamma x(\eps \cdot z)^{-1}}(w)
 -
 \sum_{[\alpha]\leq N} q_\alpha((\eps \cdot z)^{-1}) 
X^\alpha_{x} \kappa_{\sigma_2,\dot x}(w)|
 \leq C_{N} \sum_{\substack{|\alpha'|\leq \lceil N\rfloor\\ [\alpha']>N}}
(\eps \|z\|)^{[\alpha']}
\sup_{z'\in G } |X_{z'}^{\alpha'}
\kappa_{\sigma_2,\Gamma z'}(w)|.
$$
Hence, we obtain
\begin{align*}
&\|\sigma -\sum_{[\alpha]\leq N } \eps^{[\alpha]}  \Delta^\alpha \sigma_1 \, X^\alpha_M \sigma_2\|_{\cA_0}
=	\int_{G} \sup_{\dot x\in M}  |\kappa_{\dot x}(y)  
- \sum_{[\alpha]\leq N} \eps^{[\alpha]}
\kappa_{X^{\alpha_2}_M \sigma_2} * 
(q_\alpha(\cdot ^{-1}) \kappa_{\sigma_1,\dot x})
|dy
\\
&\qquad=	\int_{G} \sup_{\dot x\in M}  |\int_G (\kappa_{\sigma_2,\Gamma x(\eps \cdot z)^{-1}}( y z^{-1}) - \sum_{[\alpha]\leq N} q_\alpha((\eps \cdot z)^{-1}) 
X^\alpha_{x} \kappa_{\sigma_2,\dot x}(y z^{-1})
)\ \kappa_{\sigma_1,\dot x}(z) dz
|dy
\\
&\qquad  
\leq 
\int_{G} \sup_{\dot x\in M}  \int_G  
C_{N} \sum_{\substack{|\alpha'|\leq \lceil N\rfloor\\ [\alpha']>N}}
(\eps \|z\|)^{[\alpha']}
\sup_{\dot z'\in M } |X_{z'}^{\alpha'}
\kappa_{\sigma_2,\dot z'}(yz^{-1})|
|\kappa_{\sigma_1,\dot x}(z) |dz
 dy
\\
&\qquad  
\leq C_{N} \sum_{\substack{|\alpha'|\leq \lceil N\rfloor\\ [\alpha']>N}} \eps^{[\alpha']}
\int_{G}  \int_G  
\sup_{\dot z'\in M } |X_{z'}^{\alpha'}
\kappa_{\sigma_2,\dot z'}(yz^{-1})| \ 
 \|z\|^{[\alpha']}
\sup_{\dot x\in M} |\kappa_{\sigma_1,\dot x}(z) |dz
 dy
\\
&\qquad  \qquad = C_{N} \sum_{\substack{|\alpha'|\leq \lceil N\rfloor\\ [\alpha']>N}} \eps^{[\alpha']}
\|\sup_{\dot z'\in M } |X_{z'}^{\alpha'}
\kappa_{\sigma_2,\dot z'}\|_{L^1(G)}
\|
 \|\cdot \|^{[\alpha']}
\sup_{\dot x\in M} |\kappa_{\sigma_1,\dot x}\|_{L^1(G)}.
\end{align*}
This estimate together with Proposition \ref{prop_bddL2}  shows Part (1). 

We now prove 
Part (2) with the same arguments. 
Let $\sigma\in \cA_0$.
By \cite[Theorem 4.11]{withSoren},
$$
\Op^{(\eps)}(\sigma)^*
=
\Op^{(\eps)}(\sigma^{\eps,*})
$$
with the symbol $\sigma^{\eps,*}$
depending on $\eps\in (0,1]$ given by 
its convolution kernel  
$$
\kappa_{\sigma^{\eps,*},\cdot x}(y)
= \bar \kappa_{ \sigma, \Gamma x\eps\cdot y^{-1}}(y^{-1}) .
$$
By the Taylor estimates due to Folland and Stein, we have
$$
|\kappa_{\sigma,\Gamma x\eps\cdot y^{-1}}(w) 
-
\sum_{[\alpha]\leq N} q_\alpha(\eps\cdot y^{-1}) 
X^\alpha_x \kappa_{\sigma, x}(w)|
\leq C'_N  \sum_{\substack{|\alpha'|\leq \lceil N\rfloor\\ [\alpha']>N}} (\eps\|y\|)^{[\alpha']}
\sup_{z'\in G} |X^{\alpha'}_{z'} \kappa_{\sigma,\Gamma z'}(w)|.
$$
Hence, we have
\begin{align*}
&\|\sigma^{\eps,*}
-
\sum_{[\alpha]\leq N } \eps^{[\alpha]}
\Op^{(\eps)}( \Delta^\alpha  X^\alpha \sigma^*)\|_{\cA_0}
\\&\qquad=\int_G \sup_{\dot x\in M}
|\bar \kappa_{\sigma,\Gamma x\eps\cdot y^{-1}}(y^{-1}) 
-
\sum_{[\alpha]\leq N} q_\alpha(\eps\cdot y^{-1}) 
X^\alpha_x \bar \kappa_{\sigma, x}(y^{-1})| dy\\
&\qquad \leq \int_G \sup_{\dot x\in M}C'_N  \sum_{\substack{|\alpha'|\leq \lceil N\rfloor\\ [\alpha']>N}} (\eps\|y\|)^{[\alpha']}
\sup_{z'\in G} |X^{\alpha'}_{z'} \bar \kappa_{\sigma,\Gamma z'}(y^{-1})| dy\\
&\qquad\qquad = C'_N \sum_{\substack{|\alpha'|\leq \lceil N\rfloor\\ [\alpha']>N}} \eps^{[\alpha']}
\int_G \sup_{\dot z'\in M} |X^{\alpha'}_{z'} \bar \kappa_{\sigma,\dot z'}(y^{-1})| \|y\|^{[\alpha']} dy.
\end{align*}
This estimate together with Proposition \ref{prop_bddL2}  shows Part (2). 
\end{proof}

The analysis  shows that if $\sigma_2$ and $\sigma$ do not depend on $x$ then 
$$
\Op^{(\eps)}(\sigma_1)\Op^{(\eps)}(\sigma_2)
=
\Op^{(\eps)}(\sigma_{1}  \sigma_{2})
\qquad\mbox{and}\qquad
\Op^{(\eps)}(\sigma)^*
=
\Op^{(\eps)}(\sigma^*).
$$
Moreover, if $\sigma_1$ and $\sigma$ are the symbols of differential operators, then the sums over $\alpha$ in 
Proposition \ref{prop_symbolic_calculus}
are finite and the expansions are exact  for $N$ large enough. 

\section{Asymptotics}
\label{sec_asympt}

The semi-classical approach lends itself to understand asymptotics as the small parameter $\eps$ goes to 0.

\subsection{Estimates for kernels, Hilbert-Schmidt norms and traces}
\label{subsubsec_L2MGh}

The integral kernels enjoy the following estimates:
\begin{proposition}
	\label{prop_tr}
Let $\sigma\in \cA_0$ with associated kernel $\kappa_{\dot x}(z)$.

\begin{enumerate}
	\item The integral kernel  $K^{(\eps)}$ of  $\Op^{(\eps)}(\sigma)$ is smooth on $M\times M$ and satisfies for $\eps$ small: 
$$ 
\forall \dot x\in M\qquad
K^{(\eps)} (\dot x,\dot x) = \eps^{-Q} \kappa_{\dot x}(0) +O(\eps^\infty).
$$

\item For $\sigma\in \cA_0$,  $\Op^{(\eps)}(\sigma)$ is a Hilbert-Schmidt and trace-class operator on $M$ satisfying
$$
\tr \left(\Op^{(\eps)}(\sigma)\right)
= \eps^{-Q} \int_{M}\kappa_{\dot x}(0)  d\dot x
\ + \ O(\eps^\infty).
$$
\end{enumerate}
\end{proposition}
\begin{proof}
When the symbol $\sigma$ and the associated kernel $\kappa$ do not depend on $\dot x$, the statement boils down to the ones in \cite[Section 4.1]{nil-manifold}.
The case of a general $\sigma$ extends readily for Part (1), which then readily implies the trace property in Part (2) since 
$$
\tr \left(\Op^{(\eps)}(\sigma)\right) = \int_M K^{(\eps)} (\dot x,\dot x)d\dot x.
$$
\end{proof}

Note that the Fourier inversion formula \eqref{eq_FI} yields with the notation of Proposition \ref{prop_tr}:
$$
\kappa_{\dot x}(0) = \int_{\Gh} \tr\left( \sigma(\dot x,\pi) \right), 
\quad\mbox{so}\quad \int_{M}\kappa_{\dot x}(0)  d\dot x = \iint_{M\times \Gh} \tr\left( \sigma(\dot x,\pi) \right) d\dot x  d\mu(\pi).
$$

We open a brief parenthesis devoted to  
the tensor product of the Hilbert spaces $L^2(M)$ and $L^2(\Gh)$ defined in 
 Section \ref{subsec_aboutGnilp}:
$$
L^2(M\times \Gh) := \overline{L^2(M) \otimes  L^2(\Gh)}.
$$
We may identify $L^2(M\times \Gh)$ with the space of measurable fields of Hilbert-Schmidt operators $\sigma = \{\sigma(\dot x,\pi)\ : \ (\dot x,\pi) \in M\times \Gh\}$ such that 
$$
\|\sigma\|_{L^2(M\times \Gh)}^2 :=\iint_{M\times \Gh} \|\sigma(\dot x,\pi)\|_{HS(\cH_\pi)}^2 d\dot x d\mu(\pi)<\infty.
$$
Here $\mu$ is the Plancherel measure on $\Gh$, see Section \ref{subsec_aboutGnilp}.
The group Fourier transform yields an isomorphism between the Hilbert spaces $L^2(M\times \Gh)$ and $L^2(M\times G)$, 
and $\cF_G^{-1}\sigma (\dot x,\cdot)= \kappa_{\dot x}$ will still be called the associated kernel of $\sigma$.
The Plancherel formula yields:
$$
\|\sigma\|_{L^2(M\times \Gh)} = \|\kappa\|_{L^2(M\times G)}.
$$

Naturally $\cA_0\subset L^2(M\times \Gh)$. Proposition \ref{prop_tr} and the properties of the semi-classical calculus imply:
\begin{corollary}
\label{cor_HS}
For $\sigma\in \cA_0$,  $\Op^{(\eps)}(\sigma)$ is a Hilbert-Schmidt on $M$ satisfying
$$
\|\Op^{(\eps)}(\sigma)\|_{HS(L^2(M))}^2 
= \eps^{-Q} \|\sigma\|_{L^2(G\times M)}^2
\ + \ O(\eps).
$$
\end{corollary}

\begin{proof}
The properties of the semi-classical calculus (Section \ref{subsec_Symbolic_Calculus}) imply
\begin{align*}
\|\Op^{(\eps)}(\sigma)\|_{HS(L^2(M))}^2 
&= \tr 	\left(\Op^{(\eps)}(\sigma)\ \Op^{(\eps)}(\sigma)^*\right)
\\
&=\tr 
\left( \Op^\eps \left(  \sigma \sigma^*\right)\right) +O(\eps). 
\end{align*}
By Proposition \ref{prop_tr}, 
$$
\tr 
\left( \Op^\eps \left(  \sigma \sigma^*\right)\right)
=
\eps^{-Q}\int_{M\times \Gh} \tr \left(  \sigma \sigma^*\right) d\dot xd\mu +O(\eps^\infty).
$$
We conclude with $\int_{M\times \Gh} \tr \left(  \sigma \sigma^*\right) d\dot xd\mu=\|\sigma\|_{L^2(G\times M)}^2$. 
\end{proof}

\subsection{Non-centred quantum variance}
\label{subsec_Qvar}

In this section, we consider an orthonormal basis
 $(\varphi_j)$  of $\cR_M$-eigenfunctions for the Hilbert space $L^2(M)$ and the corresponding spectral counting function $N(\Lambda)$  as in Corollary \ref{cor_prop_RM}, and we set
$$
V_\eps(\sigma):=\frac 1{N(\eps^{-\nu})} \sum_{j : \eps^\nu \mu_j\in [0,1]}
\left|(\Op^{(\eps)} (\sigma) \varphi_j, \varphi_j)_{L^2(M)}\right|^2, \quad \sigma\in \cA_0,
	$$
	and in its lim-sup:
$$
V_0(\sigma):=\limsup_{\eps\to 0} V_\eps(\sigma).
	$$
	We call 	$V_0(\sigma)$ the (non-centred quantum) variance  associated with $(\varphi_j)$ for a symbol $\sigma\in \cA_0$.
	
\begin{remark}
\label{rem_originalmotivation}
	The motivation for considering this is that the properties of  $V_0$ imply reduced quantum ergodicity  in the case of the torus, see \cite[Section 3.1]{Nalini+Clo+Faure} and  Section \ref{subsubsec_Tn} below.
	The starting point and original motivation of this work was to understand whether the ideas could be applied in the context of nil-manifolds. 
	At least in the case of nil-manifolds of Heisentberg type, our analysis shows an obstruction, see Section \ref{subsubsec_obsgrHtype}
 below. 
\end{remark}

\subsubsection{First consequences of the semi-classical calculus}
From Proposition \ref{prop_bddL2},
we have for any $\sigma\in \cA_0$ and $\eps\in (0,1]$
$$
V_\eps(\sigma)
\leq \|\Op^{(\eps)} (\sigma)\|_{\sL(L^2(M))}
\leq \|\sigma\|_{\cA_0},
$$
Hence $V_0(\sigma)\in [0,\|\sigma\|_{\cA_0}] $ is finite.

The properties of the semi-classical calculus, see  Section \ref{subsec_Symbolic_Calculus},
 imply
 for any $\sigma\in \cA_0$
\begin{align}
[\eps^\nu \cR,\Op^{(\eps)}(\sigma)]
&=
\Op^{(\eps)}(\widehat \cR )\Op^{(\eps)}(\sigma) - \Op^{(\eps)}(\sigma \widehat \cR)
\nonumber
\\&=
\Op^{(\eps)}([ \widehat\cR, \sigma ])
+
\eps^{\upsilon_1} \sum_{[\alpha]=\upsilon_1} 
\Op^{(\eps)}(\Delta^\alpha \widehat \cR  \ X_{\dot x}^\alpha \sigma)
\ +\  O(\eps^{\upsilon_1+1}),
\label{eq_comRop}
\end{align}
where $\upsilon_1$ is the smallest weight of the dilations, see Section \ref{subsubsec_dilations}.
Therefore,
for any $\cR_M$-eigenfunction $\varphi$,  we have:
\begin{align}
0&=([\eps^\nu \cR,\Op^{(\eps)}(\sigma)]\varphi,\varphi)
\nonumber
\\&=
(\Op^{(\eps)}([ \widehat\cR, \sigma ])\varphi,\varphi)
+
\eps^{\upsilon_1} \sum_{[\alpha]=\upsilon_1} 
(\Op^{(\eps)}(\Delta^\alpha \widehat \cR  \ X_{\dot x}^\alpha \sigma)\varphi,\varphi)
\ +\  O(\eps^{\upsilon_1+1})\|\varphi\|_{L^2(M)}^2.
\label{eq_comRopphiphi}
\end{align}

The second term on the right-hand sides of  \eqref{eq_comRop} and \eqref{eq_comRopphiphi} leads us to  define the operation 
$$
\sE := \sum_{[\alpha] = \upsilon_1} \Delta^{\alpha}\widehat \cR X_{\dot x}^{\alpha},
$$
on the symbols. 
For example, we will compute the operation $\sE$ explicitely  for the intrinsic sub-Laplacians in the stratified case  in Lemma \ref{lem_compcL}.
We have obtained:
\begin{equation}
\label{eq_Opbracketvarphi}
	(\Op^{(\eps)} ([\widehat \cR ,\sigma]) \varphi,\varphi)_{L^2(M)}
=
-\eps^{\upsilon_1} (\Op^{(\eps)} (\sE \sigma) \varphi,\varphi)_{L^2(M)}+ O(\eps^{\upsilon_1+1})\|\varphi\|_{L^2(M)}^2.
\end{equation}

\subsubsection{Symbols annihilating $V_0$}
The equalities in \eqref{eq_comRop} and \eqref{eq_comRopphiphi}
 allow us to identify some symbols annihilating $V_0$:
\begin{lemma}
\label{lem_QVadR}
	If $\sigma = \ad (\widehat \cR) (\tau) = [\widehat \cR ,\tau]$  for some $\tau \in \cA_0$ then 
	$$
	V_\eps(\sigma) = O(\eps^{\upsilon_1})
	\qquad\mbox{so}\qquad V_0(\sigma) = 0.
	$$
\end{lemma}
\begin{proof}
From \eqref{eq_comRopphiphi} applied to $\tau$, we see 
	\begin{align*}
(\Op^{(\eps)} (\sigma) \varphi_j, \varphi_j)_{L^2(M)}
&= 
(\Op^{(\eps)} ([\widehat \cR ,\tau]) \varphi_j, \varphi_j)_{L^2(M)}
\\&= 
([\eps^\nu \cR_M, \Op^{(\eps)} (\tau) ]\varphi_j, \varphi_j)_{L^2(M)}+
O(\eps^{\upsilon_1}).
\end{align*}
The first term of the last right-hand side vanishes,  
and the statement follows. 
\end{proof}

Proceeding as in the proof of Lemma \ref{lem_QVadR}, we obtain:
\begin{lemma}
\label{lem_QVsEcomR}
	If $\sigma = \sE \tau$ with $\tau\in \cA_0$ commuting with $\widehat \cR$ then 
	$$
	V_\eps(\sigma) = O(\eps)
	\qquad\mbox{so}\qquad V_0(\sigma) = 0.
	$$
\end{lemma}
\begin{proof}
By \eqref{eq_Opbracketvarphi}, 
we have
	\begin{align*}
(\Op^{(\eps)} (\sigma) \varphi_j, \varphi_j)_{L^2(M)}
&= 
(\Op^{(\eps)} (\sE \tau) \varphi_j, \varphi_j)_{L^2(M)}\\
&= 
-\eps^{-\upsilon_1} (\Op^{(\eps)} ([\widehat \cR, \tau]) \varphi_j, \varphi_j)_{L^2(M)} +O(\eps) ,
\end{align*}
By hypotheses, $[\widehat \cR, \tau]=0$. The statement follows. 
\end{proof}

\subsubsection{Further consequences of the semi-classical calculus}
The properties of the semi-classical calculus already imply also the following properties: 
\begin{proposition}
\label{prop_QVL2}
\begin{enumerate}
\item We have for any $\sigma\in \cA_0$
$$
V_0 (\sigma)=V_0 (\sigma^*)
\quad\mbox{and}\qquad
V_0 (\sigma) \leq \frac 1{c} \|\sigma \|^2_{L^2(M\times \Gh)} ,
$$
with $c$  the constant from the Weyl laws (see Corollaries \ref{cor_prop_RM}).\item The map $V_0 :\cA_0 \to [0,\infty)$ is a quadratic form that 
 admits a unique continuous extension to the Hilbert space $L^2(M\times \Gh)$. 
It  satisfies the inequalities of Part (1) on $L^2(M\times \Gh)$, as well as 
\begin{equation}
\label{eq_prop_QVL2E01}
\forall \sigma\in L^2(M\times \Gh)\qquad 
V_0 (\sigma) =V_0(\widehat E[0,1] \sigma )=V_0(\widehat E[0,1] \sigma )=V_0(\widehat E[0,1] \sigma \widehat E[0,1]).
\end{equation}
Denoting by $V_0:  L^2(M\times \Gh)\times L^2(M\times \Gh) \to \bC$ the corresponding skew-hermitian map,  its kernel $\ker V_0 := \{\rho\in L^2(M\times \Gh) : V_0(\rho ,\cdot)=0\}$ is a closed subspace of $L^2(M\times \Gh) $ containing $\ad (\widehat \cR) \cA_0$ and $\sE \cA_0^{\widehat \cR}$.
 \end{enumerate}
\end{proposition}

Above and in the rest of the paper, 
we denote by 
$$
\cA_0^{\widehat \cR} := \{\sigma \in \cA_0 : \sigma \widehat \cR = \widehat \cR \sigma\}
$$
the space of symbols in $\cA_0$ commuting with $\widehat \cR$. This space will be further studied in Section \ref{subsec_symbcomR}.

\begin{proof}[Proof of Proposition \ref{prop_QVL2}]
The  properties of the semiclassical calculus implies 
$$
V_\eps ( \sigma ^*) =  V_\eps (\sigma)+O(\eps)
$$
by Proposition \ref{prop_symbolic_calculus}
while  by Corollaries \ref{cor_prop_RM} and  \ref{cor_HS}
$$
V_\eps ( \sigma )
\leq 
\frac {\|\Op^\eps(\sigma) \|_{HS(L^2(M))}^2
}{N(\eps^{-\nu})} = \frac{\eps^{-Q}\|\sigma \|_{L^2(M\times\Gh)}^2}{c \eps^{-Q}} +O(\eps),
$$
Part (1) follows by passing to the lim-sup.

 We observe that $V_\eps$ and $V_0$ are quadratic forms corresponding to the skew-hermitian maps $\cA_0\times \cA_0 \to \bC$ given for $\sigma_1,\sigma_2\in \cA_0$ by:
\begin{align*}
V_\eps(\sigma_1,\sigma_2)&:=\frac 1{N(\eps^{-\nu})} \sum_{j : \eps^\nu \mu_j\in [0,1]}
(\Op^{(\eps)} (\sigma_1) \varphi_j, \varphi_j)_{L^2(M)}
\overline{(\Op^{(\eps)} (\sigma_2) \varphi_j, \varphi_j)_{L^2(M)}},\\	
V_0(\sigma_1,\sigma_2)&:=\limsup_{\eps\to 0}
V_\eps(\sigma_1,\sigma_2).
\end{align*}

As $\cA_0\subset L^2(M\times \Gh)$ is dense in the Hilbert space 
$L^2(M\times \Gh)$, and $
V_0 (\cdot ) \leq c^{-1} \|\cdot \|^2_{L^2(M\times \Gh)}$ on $\cA_0$, $V_0$ extends uniquely into a continuous quadratic form on $L^2(M\times \Gh)$ satisfying $
V_0 (\cdot ) \leq c^{-1} \|\cdot \|^2_{L^2(M\times \Gh)}$.

Let $\psi  \in \cD(\bR)$ with $\psi\equiv 1$ on $[0,1]$.
When $\eps^\nu\mu_j\in [0,1]$ and $\sigma\in \cA_0$, we have
$$
V_\eps(\sigma)=V_\eps (\psi(\widehat \cR))\ \sigma) =V_\eps (\sigma \psi(\widehat \cR))).
$$
By taking a lim-sup and density of $\cA_0$ in $L^2(M\times \Gh)$, this implies 
$$
\forall \sigma\in L^2(M\times \Gh)\qquad 
V_0(\sigma)=V_0 (\psi(\widehat \cR))\ \sigma) =V_0(\sigma \psi(\widehat \cR))).
$$
This together with the Cauchy-Schwartz inequality 
imply that for any $\eta\in(0,1]$, 
\begin{align*}
|V_0(\widehat E[0,1] \sigma )	- V_0(\sigma)|
&\leq 
4 c^{-1} \|(\widehat E[-\eta,0]+\widehat E[1,1+\eta]) \sigma\|_{L^2(M\times \Gh)}
\|\sigma \|_{L^2(M\times \Gh)}.
\end{align*}
The right-hand side goes to 0 as $\eta\to 0$ by Lebesgue's theorem of dominated convergence. This shows $V_0 (\sigma) =V_0(\widehat E[0,1] \sigma)$ for any $\sigma\in L^2(M\times \Gh)$. We proceed in the same way to obtain 
$V_0 (\sigma) =V_0( \sigma\widehat E[0,1])$, showing \eqref{eq_prop_QVL2E01}.

Now let $\sigma,\tau\in \cA_0$.
The equality \eqref{eq_Opbracketvarphi}
yields
$$
V_\eps (\sigma, [\widehat \cR, \tau])	
 =-\eps^{\upsilon_1} V_\eps (\sigma, \sE\tau)	
 +O(\eps)^{\upsilon_1 +1}.
 $$
Taking the limsup, 
we obtain 
$$
	V_0(\sigma, [\widehat \cR, \tau])=\limsup_{\eps\to 0} -\eps^{\upsilon_1} V_\eps (\sigma, \sE\tau) =0,
$$
and if $\tau$ commutes with $\widehat \cR$ 
$$
	V_0 (\sigma, \sE\tau) 
=\limsup_{\eps\to 0} V_\eps (\sigma, \sE\tau) 
= \limsup_{\eps\to 0}-V_\eps (0)
=0.
$$
In other words, the two equalities above show respectively 
 $$
 \ad (\widehat \cR) \cA_0 \subset \ker V_0
 \qquad\mbox{and}\qquad
\sE \cA_0^{\widehat \cR}\subset \ker V_0.
$$
This concludes the proof. 
\end{proof}

\begin{remark}
Proposition \ref{prop_QVL2} and its proof imply that any operator $A\in \sL(L^2(M\times \Gh))$ satisfying 
$$
\forall \sigma_1,\sigma_2\in L^2(M\times \Gh)\qquad 
V_0(\sigma_1,A\sigma_2) = V_0 (A\sigma_1,\sigma_2),
$$
leads to a centred quantum variance
given via
$$
V_0(\sigma - A \sigma) =V_0(\sigma) - V_0(A \sigma).
$$
Examples of such $A$'s are an orthogonal projection 
onto a closed subspace of $\ker V_0$ or of $(\ker V_0)^\perp$; these examples are rather trivial in the sense that we would have  $V_0(A\sigma)=0$ for the former and $V_0(\sigma -A\sigma)=0$ for the latter.	
\end{remark}

 \subsection{Case of the torus}
 \label{subsubsec_Tn}
 
 Here, we  show that 
Proposition \ref{prop_QVL2} implies easily the reduced quantum regodicity in the case of the flat torus:
 \begin{theorem}
\label{thm_QETn}
Let $(\varphi_j)_{j\in \bN}$ be an orthonormal basis of $L^2(\bT^n)$ for the canonical Laplacian $\Delta=-\partial_1^2-\ldots -\partial_n^2$:
$$
\Delta\varphi_j = \mu_j \varphi_j, \quad 0=\mu_0<\mu_1 \leq \mu_2\leq \ldots .
$$ 
Then there exists a subsequence $(j_k)\subset \bN$ of density 1 such that 
$$
\forall f\in C(\bT^n)\qquad 
\lim_{k\to \infty} \int_M f(\dot x) |\varphi_{j_k}(\dot x)|^2 d\dot x = \int_M f(\dot x)d\dot x.
$$
\end{theorem}
 
 This is a sophisticated version of the  proof given in \cite[Section 3.1]{Nalini+Clo+Faure} given by more elementary means.

\subsubsection{Determining $(\sE \cA_0)^\perp$ in the torus case}
First, we apply what precedes to the case where $\cR = \Delta$ is the canonical Laplacian of $\bT^n$.
The $\cA_0$-symbols are the functions $\bT^n\times \bR^n \ni (\dot x,\xi) \mapsto \sigma(\dot x,\xi)$ in $C^\infty (\bT^n;\cS(\bR^n))$.
We compute easily 
$$
\sE = -2\sum_j\partial_{\dot x_j} \partial_{\xi_j}
\qquad\mbox{and}\qquad  
\ad (\widehat \cL) \cA_0=\{0\},\qquad  \sE \cA_0^{\widehat \cL} =\sE \cA_0,
$$
 since  
all the symbols commute with each other in this setting.
Moreover, $(\sE \cA_0)^\perp$
is the space of symbols $\sigma \in L^2(\bT^n;\cS(\bR^n))$ viewed as functions on $\bT^n\times \bR^n$ and satisfying  $\sE \sigma = 0$ in the sense of distributions on $\bT^n\times \bR^n$.
 We can determine the projection onto $(\sE \cA_0)^\perp$ 
using the well-known properties of the Fourier series and of the Euclidean Fourier transforms:
 \begin{lemma}
 \label{lem_thm_QETn}
 	With the setting of Theorem \ref{thm_QETn},
 	the orthogonal projection onto $(\sE \cA_0)^\perp\subset L^2(M\times \Gh)$ is given by 
 	 	$$
 \int_{\bT^n} \sigma, 
 	\qquad \mbox{where}\quad 
 	\left (\int_{\bT^n} \sigma\right )(\dot x,\xi) = \int_{\bT^n} \sigma(\dot x' ,\xi) d\dot x', \quad \sigma\in L^2(\bT^n\times \bR^n). 
 	$$
  \end{lemma}
\begin{proof}[Proof of Lemma \ref{lem_thm_QETn}]
If $\sigma\in L^2(\bT^n\times \bR^n)$ satisfies $\sE \sigma=0$ in the sense of distributions, then 
applying  the Fourier transform on $\bT^n\times \bR^n$
to $\sigma$ and denoting the resulting function as $\widehat \sigma (\ell,y)$, $\ell\in \bZ^n,$ $y\in \bR^n$, we see that 
$$
\forall \ell\in \bZ^n,\ y\in \bR^n\qquad 
(\sum_j \ell_j y_j) \widehat \sigma (\ell,y)=0,
$$
so $\widehat \sigma (\ell,\cdot )\equiv 0$ for any $\ell\in \bZ^n \setminus \{0\}$. 
This shows that $(\sE \cA_0)^\perp$  is the space of functions $\sigma\in L^2(\bT^n\times \bR^n)$ which are constant in $\dot x\in \bT^n$. The conclusion follows. 
\end{proof}

\subsubsection{Proof of Theorem \ref{thm_QETn}}
We can now  deduce the reduced QE from Proposition \ref{prop_QVL2}:

By Lemma \ref{lem_thm_QETn}, for any $\sigma\in L^2(\bT^n\times \bR^n)$, 
$$
\sigma-\int_{\bT^n}\sigma \in  
 \overline{\sE \cA_0}\subset \ker V_0
$$
 by Proposition \ref{prop_QVL2} (2), 
so $V_0 (\sigma-\int_{\bT^n}\sigma)=0$.
We apply this to $\sigma(\dot x,\xi) = f(\dot x)\psi(|\xi|^2)$
 for any 
$f\in L^2(\bT^n)$, having fixed a function  $\psi\in \cD(\bR)$ with  $\psi\equiv 1$ on $[0,1]$. We have
$$
V_0\left  (\sigma - \int_{\bT^n} \sigma \right )=\lim_{\eps\to 0}
\frac 1{N(\eps^{-\nu})} \sum_{j : \eps^\nu \lambda_j\in [0,1]} \left|
\int_{\bT^n} f(\dot x) |\varphi_j(\dot x)|^2 d\dot x - \int_{\bT^n} f(\dot x)  d\dot x\right|^2 =0
$$
The end of the proof is now a classical argument relying on the separability of the space $C(\bT^n)$ of continuous functions on $\bT^n$. Indeed, 
the null limit above together with a well-known result on Cesaro means of positive sequences, see e.g. \cite[Theorem 1.8]{Walters} (or equivalently with  a  direct proof `by hand' 
\cite[p.111]{Nalini+Clo+Faure}) yield the existence of a subsequence $(j_k)$ of density 1 such that 
$$
\lim_{j=j_k\to \infty} \int_{\bT^n} f(\dot x) |\varphi_j(\dot x)|^2 d\dot x = \int_{\bT^n} f(\dot x)d\dot x.
$$
However, this subsequence $(j_k)$ depends on $f$.
We apply the above result to each continuous function $f=f_\ell$, $\ell\in \bN_0$, of a dense countable family $(f_\ell)$ of $C(\bT^n)$.
We then   extract diagonally a new subsequence $j_k$ still of density one and with respect to which  convergence holds for any of the functions $f_\ell$, and therefore also
for any continuous function $f$. 

\subsubsection{Obstruction in the case of groups of Heisenberg type}
\label{subsubsec_obsgrHtype}
In the case of the torus, we have found an 
an operator $B\in \sL(L^2(M\times \Gh))$ satisfying for $\sigma = f \psi (\widehat \cR))$ with $f\in L^2(\bT^n)$, $\psi\in \cD(\bR)$ and $\psi=1$ on $[0,1]$
\begin{itemize}
	\item $V_0(\sigma-B\sigma)=0$ (in fact, this is satisfied for all $\sigma\in L^2(M\times \Gh)$,
	\item $(\Op^{(\eps)} B\sigma, \varphi_j,\varphi_j)$ is independent of $\eps$ and has a limit as $j\to \infty$ (in fact, this quantity  is independent of $j$ as well). 
\end{itemize} 
This operator $B$ was the orthogonal projection onto $(\sE\cA_0)^\perp$ given by  $\int_{\bT^n} $.

\smallskip

In the case of  groups of Heisenberg type with $\cR$ being the intrinsic sub-Laplacian $\cL$, we have the following inclusions:
$$
	\overline{\sE \cA_0^{\widehat \cL}} \subset \overline{ \ad (\widehat \cL) \cA_0 } \quad\mbox{in}\quad L^2(M\times \Gh).
$$
This is a direct consequence of \cite[Lemma 4.1 (2)]{FFJST}. 
Therefore, Proposition \ref{prop_QVL2} only yields that 
$\ker V_0$ contains the closure of $\ad (\widehat \cL) \cA_0$, or equivalently, that $(\ker V_0)^\perp $ is included in the closure of $\cA_0^{\widehat \cL}$.
This is not enough to determine an operator $B$ as above, 
as the closure of $\cA_0^{\widehat \cL}$ is too large for its orthogonal projector to be described in a way as simple as in the torus.  We will come back to this at the end of the paper. 
 
 \section{Semi-classical limits}
\label{sec_sclim}

An important advantage of our symbolic approach is that it yields a precise description of the limit as $\eps\to 0$ of quadratic quantities $(A \phi_\eps,\phi_\eps)$ for an operator (more precisely, a family of operators) $A=A^\eps $ in the semi-classical calculus and a family of functions $(\phi_\eps)_\eps$  in $L^2(M)$. 
This semi-classical limit is expressed as measures which are operator valued; this is due to the non-commutativity of our setting.

\subsection{Quadratic limits and states of $C^*(M\times \Gh) $}

We consider 
$$
C^*(M\times \Gh)  = \overline{\cA_0}^{\|\cdot\|_{L^\infty(M\times \Gh)}}
$$
the closure of $\cA_0$ for the norm 
$\|\cdot\|_{L^\infty(M\times \Gh)}$ defined in \eqref{eq_Linftysigma}.
Let us summarise its main properties as a $C^*$-algebra:
\begin{proposition}
\label{prop_algcA}
\begin{enumerate}
\item The space $C^*(M\times \Gh)$ is a separable non-unital type I $C^*$-algebra.	
	\item Let $\kappa\in \cS(G)$ with $\kappa\geq 0$ and $\int_G \kappa(y)dy=1$. Set $\kappa^{(t)} (y) = t^{-Q} \kappa(t\cdot y)$. 
Then the family of symbols  
	$\sigma_{\kappa,t}$ in $\cA_0$ given by $\sigma_{\kappa,t}(\dot x,\pi)  = \pi(\kappa^{(t)}) = t \cdot \pi (\kappa)$  is an approximation of the identity for $C^*(M\times \Gh)$ as $t\to 0$.
	This is also so for any subsequence $(\sigma_{t_j})_{j\in \bN}$ with $\lim_{j\to \infty}t_j \to 0$.
\end{enumerate}
	\end{proposition}
	
	The spectrum of the $C^*$-algebra $C^*(M\times \Gh)$ and its von Neumann algebra $L^\infty(M\times \Gh)$ will be described in 
	Sections \ref{subsec_dualC*MGh} and \ref{subsubsec_LinftyMGh} respectively.

\begin{proof}[Proof of Proposition \ref{prop_algcA}]
The space $\cA_0$ is the image via the group Fourier transform of $C^\infty (M \colon \cS(G))$.
It may be described as the closure of the algebraic tensor of $\cD(M)$ and $\cF_G \cS(G)$ for the corresponding topology.
We observe that $\|\cdot\|_{L^\infty(M\times \Gh)}$ is continuous for this topology, and that 
$\cD(M)$ and $\cF_G \cS(G)$ are dense in the  $C^*$ algebras $C(M)$ and $C^*(G)$ respectively.
Therefore,  $C^*(M\times \Gh) $ is the tensor product product of $C(M)$ with $C^*(G)$; as $C(M)$ is commutative, the min or max tensor products coincide.
As $C(M)$ is a separable and $C^*(G)$ is separable and type I  (see Section \ref{subsubsec_vNG}), 
 $C^*(M\times \Gh) $ is separable and type I, implying  Part (1).
 
By Proposition \ref{prop_app_id} (2) for $p=1$, we have 
$$
\lim_{t\to 0}\|\sigma_{\kappa,t} \tau - \tau\|_{L^\infty(M\times \Gh)}=0=\lim_{t\to 0}\| \tau\sigma_{\kappa,t} - \tau\|_{L^\infty(M\times \Gh)}.
$$
first for any $\sigma\in \cA_0$, and then for any $\sigma\in C^*(M\times \Gh)$ by density. 
This shows Part (2). 
\end{proof}

The states (i.e. the continuous linear positive forms  of norm 1) of the $C^*$-algebra $C^*(M\times \Gh) $ are useful when describing the following limits:
\begin{proposition}
\label{prop_scL}
	Let $(\phi_\eps)_{\eps\in (0,1]}$ be a bounded family  in $L^2(M)$.
Consider the associated linear functionals $\ell_\eps$, $\eps\in (0,1]$, on $\cA_0$ given by:
\begin{equation}
\label{eq_elleps}
\ell_\eps(\sigma) = \left(\Op^{(\eps)}(\sigma) \phi_\eps,\phi_\eps\right)_{L^2(M)}, \qquad \sigma\in \cA_0. 	
\end{equation}

\begin{enumerate}
\item For each $\sigma\in \cA_0$, 
$$
\forall \eps\in (0,1]\qquad
|\ell_{\eps} (\sigma)|\leq \|\sigma\|_{\cA_0} \sup_{\eps'\in (0,1]}\|\phi_{\eps'}\|_{L^2(M)}^2.
$$
\item We may 
extract a subsequence, denoted by $(\eps_k)_{k\in \bN}$, such that the limits of $\|\phi_{\eps_k}\|_{L^2(M)}$ and of $(\ell_{\eps_k}(\sigma))_{k\in \bN}$ for each $\sigma\in \cA_0$ exists:
 $$
 \exists \lim_{k\to \infty} \|\phi_{\eps_k}\|_{L^2(M)} :=c_0
 \qquad\mbox{and}\qquad 
 \forall \sigma\in \cA_0\qquad 
 \exists \lim_{k\to \infty} \ell_{\eps_k}(\sigma) =: \ell_0(\sigma).
 $$
This defines a linear map $\ell_0:\cA_0\to \bC$
satisfying $|\ell_0(\sigma)|\leq \|\sigma\|_{\cA_0} c_0^2$. 

\item The map $\ell_0$
extends uniquely and naturally to a continuous positive linear functional  still denoted by $\ell_0:C^*(M\times \Gh) \to \bC$.
If  $c_0=0$, then $\ell_{0}=0$.
If $c_0\neq 0$, then  $c_0^{-2}\ell_{0}$ is a state of the $C^*$ algebra  $C^*(M\times \Gh) $.
\end{enumerate}	
\end{proposition}

For instance, the case of $\phi_\eps= 1/\sqrt{\vol(M)}$ being a constant function on $M$ is easily determined. Indeed, we compute easily for any $\sigma\in \cA_0$ with convolution kernel $\kappa_{\dot x}$:
\begin{align*}
	\Op^{(\eps)} \phi_\eps  (\dot x) 
	&= (\vol(M))^{-1/2} \int_G \kappa_{\dot x} (z)dz, 
\\
\left(\Op^{(\eps)}(\sigma) \phi_\eps ,\phi_\eps \right)_{L^2(M)} 
&= (\vol(M))^{-1}\iint_{M\times G} \kappa_{\dot x}(z)dz d\dot x = (\vol(M))^{-1}
\int_M \sigma(\dot x, 1_{\Gh}) d\dot x.
\end{align*}
Hence in this case, the corresponding state of $C^*(M\times \Gh) $ is given by 
$$
\ell_0(\sigma)= (\vol(M))^{-1} \int_M \sigma(\dot x, 1_{\Gh}) d\dot x.
$$

Results similar to Proposition \ref{prop_scL} were obtained in \cite[Sections 5 and 6]{FFPisa}, \cite[Section 3.2]{FFJST}, \cite[Proof of Theorem 4.1]{FFchina} and heuristically in \cite{bologna}. Here, we give 
a  very detailed proof, especially for Part (3) in the next section.

\begin{proof}[Beginning of the proof of Proposition \ref{prop_scL}]
Part (1) follows readily from Proposition \ref{prop_bddL2}.	
Hence, for each $\sigma\in \cA_0$, $(\ell_\eps (\sigma))_{\eps\in (0,1]})$ is bounded and  we may extract a sequence  $(\eps_k)_{k\in \bN}$ going to 0 as $k\to \infty$ such that the sequence $(\ell_{\eps_k} (\sigma))_{k\in \bN}$ converges. 
By a diagonal extraction argument, we may assume that these sequence $(\eps_k)_{k\in \bN}$  is such that 
$\lim_{k\to \infty} \ell_{\eps_k}(\sigma)$ exists for every $\sigma$ in a countable subset $S$ of $\cA_0$ that is dense in $(\cA_0,\|\cdot\|_{\cA_0})$ as well as for $\lim_{k\to \infty} \|\phi_{\eps_k}\|_{L^2(M)}$.
This implies with a $3\epsilon$-argument and Part (1) that for this same subsequence  $(\eps_k)_{k\in \bN}$ and any $\sigma\in \cA_0$, 
the sequence $(\ell_{\eps_k}(\sigma))_{k\in \bN}$ is Cauchy, and we denote its limit by 
$\ell_0(\sigma)= \lim_{k\to \infty} \ell_{\eps_k}(\sigma):=\ell_0(\sigma)$.
This  defines a linear map $\ell_0:\cA_0\to \bC$ independently of a choice of subset $S$
and  satisfying $|\ell_0(\sigma)|\leq \|\sigma\|_{\cA_0} c_0^2$ for any $\sigma\in \cA_0$.
Part (2) is proved.

For Part (3), we will need  to show that $\ell_0$ extends naturally to $C^*(M\times \Gh)$ 
as a positive  functional.
We can already prove the positivity.
Indeed,  for any $\sigma\in \cA_0$, we have
$$
 \ell_{\eps} (\sigma \sigma^*)
= \left(\Op^{(\eps)}(\sigma\sigma^*) \phi_\eps,\phi_\eps\right)_{L^2(M)}
=\left(\Op^{(\eps)}(\sigma)\Op^{(\eps)}(\sigma)^* \phi_\eps,\phi_\eps\right)_{L^2(M)}
+O(\eps), 
$$ 
by the properties of the functional calculus (see Proposition \ref{prop_symbolic_calculus}), 
and 
$$
\left(\Op^{(\eps)}(\sigma)\Op^{(\eps)}(\sigma)^* \phi_\eps,\phi_\eps\right)_{L^2(M)} = 
\|\Op^{(\eps)}(\sigma)^* \phi_\eps\|_{L^2(M)} ^2\geq 0.
$$
Hence $\ell_0:\cA_0\to \bC$ satisfies 
$\ell_0(\sigma \sigma^*)\geq 0$
 for any $\sigma\in \cA_0$.
\end{proof}

\subsection{Proof of Proposition \ref{prop_scL} (3)}

In this section, we finish the proof of Proposition \ref{prop_scL}.
It will utilise  properties of symbols in tensor forms 
as well as of the von Neumann algebra $L^\infty (M\times \Gh)$.

\subsubsection{Symbols in tensor forms}
\label{subsubsec_symbtensor}
The following technical lemma and its proof show how to approximate given symbols by symbols in tensor forms:
\begin{lemma}
\label{lem_symbtensor}
\begin{enumerate}
	\item 	Let $\sigma\in C^*(M\times \Gh) $. 
	For any $\eta>0$, there exists a symbol $\tau_\eta\in \cA_0$ 
	such that $\|\sigma -\tau_\eta\|_{L^\infty(M\times \Gh)}<\eta$. 
	Moreover, we can construct $\tau_\eta$ of the form 
	$\tau_\eta = \sum_{m=1}^N \psi_m \widehat \kappa_m$ for some integer $N\in \bN$, 
functions $\psi_m\in C^\infty (M;[0,1]),\kappa_m\in \cS(G)$, $m=1,\ldots,N$, 
satisfying 
$$
\sum_{m=1}^N \psi_m =1 \ \mbox{on}\ M
\quad\mbox{and} \quad 
\|\widehat \kappa_m \|_{L^\infty(\Gh)}< \frac \eta 2 + \|\sigma\|_{L^\infty(M\times \Gh)}.
$$
\item Let $\sigma_1,\sigma_2\in \cA_0$ and let $\eta>0$.
	There exist symbols $\tau_1,\tau_2\in \cA_0$ such that $\|\sigma_i - \tau_i\|_{\cA_0}<\eta$, $i=1,2$. Moreover, we can construct them of the form $\tau_i=\sum_{m=1}^N \psi_m 
	\sigma_i(\dot x_m, \cdot)$ for some  integer $N\in \bN$, points $\dot x_1, \ldots, \dot x_m$
 in $M$ and functions $\psi_1, \ldots, \psi_m$ in $C^\infty (M;[0,1])$ 	satisfying
	$\sum_{m=1}^N \psi_m =1$ on $M$.
\end{enumerate}
In both cases, we may assume that for any $m_0=1,\ldots,N$, 
the number of $m'=1,\ldots,N$ such that $\supp\, \psi_{m'}$ 
intersect  $\supp\,\psi_{m_0}$ is bounded by a constant depending only on $M$. We may also assume that $\psi_m=\theta_m^2$ for some $\theta_m\in C^\infty (M;[0,1])$.
	\end{lemma}
	
	The proof is classical and given for the sake of completeness. 
\begin{proof}[Proof of Lemma \ref{lem_symbtensor}] 
As it is easier to describe open sets and covering in terms of a distance, we  equipp  the smooth manifold $M$ with a Riemannian structure. To fix the idea, we consider a fix a scalar product on $\fg$; this induces a left-invariant Riemannian metric on $G$ and then a Riemannian metric on $M$.
We denote by $d$ the Riemannian distance and by $B(\dot x,r)=\{\dot y\in M : d(\dot x,\dot y)<r\} $ the open ball about $\dot x$ of radius $r>0$.
As the smooth Riemannian manifold $M$ is compact and of dimension $n$, for any $r\in (0,1)$, the (Riemannian) volume  of $B(\dot x,r)$ is comparable to $r^n$:
\begin{equation}
\label{eq_Rvolball}
	\exists c_1,c_2>\qquad \forall r\in (0,1]\qquad 
c_1r^n \leq |B(\dot x,r)| \leq c_2r^n.
\end{equation}

We fix $\sigma\in C^*(M\times \Gh) $ and $\eta>0$.
As the function $\dot x\mapsto \sigma(\dot x, \cdot)$, $M\to L^\infty(\Gh)$ is continuous, 
it is also equicontinuous and there exists $\delta>0$ such that 
$$
\forall \dot x,\dot x'\in M
\qquad d(\dot x,\dot x')<\delta
\ \Longrightarrow \ 
\|\sigma(\dot x,\cdot)-\sigma(\dot x',\cdot)\|_{L^\infty(\Gh )}<\eta/2.
$$
As $M$ is compact, we extract a finite cover from $M=\cup_{\dot x\in M} B(\dot x,\delta/2)$:
there exist  points $\dot x_1,\ldots,\dot x_m$ in $M$ 	such that $M=\cup_{m=1}^N B(\dot x_m,\delta/2)$. 
We then construct a finite subordinate partition of the identity: we start with  functions $\tilde \psi_1, \ldots, \tilde \psi_m$ in $C^\infty(M;[0,1])$ such that $\tilde \psi_m=1$ on $B(\dot x_m,\delta/2)$ but $\tilde \psi_m=0$ outside $B(\dot x_m,\delta)$, and then set 
$\psi_m := \tilde \psi_m /(\sum_{m'=1}^N \tilde \psi_{m'})$  for $m=1,\ldots, N$.
By density of $\cF_G \cS(G)$ in $L^\infty (\Gh)$, 
there exists $\kappa_m\in \cS(G)$ such that 
$\|\widehat \kappa_m -\sigma(\dot x_m,\cdot)\|_{L^\infty(\Gh)} <\eta/2$.
We check readily $\sum_m \psi_m=1$, and 
\begin{align*}
&\|\sigma - \sum_{m=1}^N \psi_m \widehat \kappa_m\|_{L^\infty(M\times \Gh )}
 \leq \sum_{m=1}^N  
\|   \psi_m (\sigma -\sigma(\dot x_m,\cdot))\|_{L^\infty(M\times \Gh )}
+ \| \sum_{m=1}^N \psi_m (\sigma(\dot x_m,\cdot) -\widehat \kappa_m)\|_{L^\infty(M\times \Gh )}	\\
&\qquad \leq \max_{m=1,\ldots,N} \sup_{\dot x \in B(\dot x_m, \delta)} \|\sigma(\dot x, \cdot) -\sigma(\dot x_m,\cdot)\|_{L^\infty(\Gh)} + \max_{m=1,\ldots,N} \|\sigma(\dot x_m,\cdot)-\widehat \kappa_m \|_{L^\infty(\Gh)} 
< \frac \eta 2+\frac \eta 2 = \eta.
\end{align*}
This shows Part (1).

For Part (2), 
we first fix a quasinorm $\|\cdot\|$ on $G$ and set the constant 
$$
\int_G (1+\|y\|)^{-(Q+1)}dy:=c'\in (0,\infty) .
$$
By continuity of  $\dot x \mapsto \kappa_{\sigma,\dot x}(1+\|\cdot\|)^{-(Q+1)}$, $M\to L^\infty(G)$, on the compact manifold $M$, 
there exists $\delta>0$ such that for $i=1,2$, we have:
$$
\forall \dot x,\dot x'\in M
\qquad d(\dot x,\dot x')<\delta
\ \Longrightarrow \ 
\|(1+\|\cdot\|)^{-(Q+1)}(\kappa_{\sigma_i,\dot x}-\kappa_{\sigma_i, \dot x'})\|_{L^\infty(G )}<\eta/c'.
$$
We extract a finite cover for the compact manifold $M=\cup_{\dot x} B(\dot x,\delta/2)=\cup_{m=1}^N B(\dot x_m,\delta/2)$, 
and a finite subordinate partition of unity $\psi_m$, $m=1,\ldots, N$ as above
from the balls $B(\dot x_m,\delta/2)$ and $B(\dot x_m,\delta)$. 
We have for $i=1,2$
\begin{align*}
\|\sigma_i - \tau_i\|_{\cA_0}
& \leq  \max_{m=1,\ldots,N}  \int_G \sup_{\dot x \in B(\dot x_m, \delta)} |\kappa_{\sigma_i,\dot x}(y) - \kappa_{\sigma_i,\dot x_m}(y)|dy
 \\&<  \frac{\eta}{c'}\int_G (1+\|y\|)^{-(Q+1)}dy =\eta. 
\end{align*}
This proves Part (2).

To show the last part of the statement, we need to amend the proofs above. 
We may assume that $\psi_m=\theta_m^2$ and $\delta\leq 1$. 
Furthermore, instead of considering a finite cover from $M=\cup_{\dot x\in M} B(\dot x,\delta/2)$, we first consider a maximal family of disjoint balls $B(\dot x',\delta/4)$, $\dot x'\in X$. By maximality, we obtain a cover $M=\cup_{\dot x'\in X} B(\dot x',\delta/2)$ from which we extract a finite cover as above but with balls whose centres are at distance at least $\delta/2$ from each other. 
Fixing one centre $\dot x'_0$, 
we consider all the balls $B(\dot x'_{m_j},\delta/2)$, $j = 1,\ldots,J$, intersecting  $B(\dot x'_0,\delta/2)$; then 
$$
B(\dot x'_0,\delta) \supset \sqcup_{j=1}^J B(\dot x'_{m_j},\delta/4), 
\qquad \mbox{so}\qquad
c_1 \delta^n \geq |J| c_2 \left(\frac \delta 4\right)^n,
$$
by \eqref{eq_Rvolball}.
Hence $|J| \leq c_1c_2^{-1}4^{-n}$.
This concludes the proof. 
\end{proof}

The next statement and its proof give bounds in terms of $\|\cdot\|_{L^\infty(\Gh)}$ for quadratic expression of  symbols in tensor forms as in Lemma \ref{lem_symbtensor} using properties of the semiclassical calculus:
\begin{lemma}
\label{lem_symbtensorquad}
	Let $\tau\in \cA_0$ be of the form $\tau =\sum_{m=1}^N \psi_m \widehat \kappa_m $
	for some integer $N\in \bN$, 
functions $\psi_m\in C^\infty (M;[0,1]),\kappa_m\in \cS(G)$, $m=1,\ldots,N$, 
satisfying 
$\sum_{m=1}^N \psi_m =1 \ \mbox{on}\ M$, 
$\psi_m=\theta_m^2$ for some $\theta_m\in C^\infty (M;[0,1])$.
Then there exists a constant $C>0$ such that 
for any $\eps\in (0,1]$ and $f\in L^2(M)$, we have
	$$
 |(\Op^\eps (\tau_\eta) f,f)_{L^2(M)}|\leq 
\left( C_0 \max_{m=1,\ldots, N}\|\widehat \kappa_m\| _{L^\infty(\Gh)} +   C\eps \right )\|f\|^2_{L^2(M)},
	$$
	where $C_0$ is the maximum over 
	for each $m_0=1,\ldots,N$ of
the number of $m'=1,\ldots,N$ such that $\supp\, \psi_{m'}$ 
intersect  $\supp\,\psi_{m_0}$.
\end{lemma}
\begin{proof}[Proof of Lemma \ref{lem_symbtensorquad}]
	The properties of the calculus implies for any $\psi\in C^\infty(M)$ and $\kappa\in \cS(G)$,
$$
	(\Op^\eps (\psi \widehat \kappa ) f,f)_{L^2(M)}
	= ( (f_G * \kappa^{(\eps)})_M, \bar \psi f)_{L^2(M)},
	\qquad \mbox{where}\qquad \kappa^{(\eps)} (y) = \eps^{-Q}\kappa_m(\eps^{-1}y),
	$$
while  by the Cauchy-Schwartz inequality and Lemma \ref{lem_RegRep},
$$
	|(\Op^\eps (\psi \widehat \kappa ) f,f)_{L^2(M)}|
	\leq \|\widehat \kappa^{(\eps)}\|_{L^\infty(\Gh)}
	\|f\|^2_{L^2(M)},
	\qquad \mbox{with}\qquad \|\widehat \kappa^{(\eps)}\|_{L^\infty(\Gh)}=\|\widehat \kappa\|_{L^\infty(\Gh)}.
	$$	
	
	With the notation of the statement,
the properties of the calculus also implies
	\begin{align*}
	(\Op^\eps (\tau_\eta) f,f)_{L^2(M)}
	&= \sum_{m,m'=1}^N(\Op^\eps (\psi_m \theta_{m'}^2\widehat \kappa_m) f,f)_{L^2(M)}
	\\&= \sum_{m,m'=1}^N(\Op^\eps ( \widehat \kappa_m) \theta_{m'} f, \theta_{m'} \psi_m f)_{L^2(M)} +O(\eps)\|f\|^2_{L^2(M)}.
	\end{align*}
	With the remarks above,   we obtain
	$$
|(\Op^\eps ( \widehat \kappa_m) \theta_{m'} f, \theta_{m'} \psi_m f)_{L^2(M)} |
	\leq \|\widehat \kappa_m\| _{L^\infty(\Gh)} 
\|\theta_{m'} f\|_{L^2(M)}\|\theta_{m'}\psi_m f\|_{L^2(M)} ,
$$
and consequently,
$$
\sum_{m,m'=1}^N|(\Op^\eps ( \widehat \kappa_m) \theta_{m'} f, \theta_{m'} \psi_m f)_{L^2(M)}|
\leq 
\max_{m=1,\ldots, N}\|\widehat \kappa_m\| _{L^\infty(\Gh)}
C_0 \sum_{m'=1}^N \|\theta_{m'}f\|^2_{L^2(M)}
$$
The conclusion follows from $\sum_{m'=1}^N\|\theta_{m'} f\|_{L^2(M)}^2=\|f\|_{L^2(M)}^2$.
\end{proof}

We can now show part of  Proposition \ref{prop_scL} (3), more precisely  the natural extension of $\ell_0$ to a positive continuous functional on $C^*(M\times \Gh)$.

\begin{proof}[Beginning of the proof of Proposition \ref{prop_scL} (3)]
For each (fixed) $\sigma\in C^*(M\times \Gh) $, 
we make  the following two observations:
\begin{itemize}
	\item Observation 1: For each $\eta>0$, considering $\tau_\eta$ constructed via Lemma \ref{lem_symbtensor} (1), we have by Lemma \ref{lem_symbtensorquad}:
	$$
	|\ell_\eps(\tau_\eta)| \leq C_0 \left ( \|\sigma\|_{L^\infty(M\times \Gh)}+ \frac \eta 2  +O(\eps)\right )\|\phi_\eps\|_{L^2(M)}^2,
	$$
	so taking the limit as $\eps=\eps_k$, $k\to\infty$:	
	$$
	|\ell_0(\tau_\eta)| \leq  c_0^2( \|\sigma\|_{L^\infty(M\times \Gh)}+ \frac \eta 2 ).$$ 
	\item Observation 2: If $(\sigma_{1,m})_{m\in \bN}$ and $(\sigma_{2,m})_{m\in \bN}$ are two sequences in $\cA_0$, then 
$$
	|\ell_0(\sigma_{1,m})-\ell_0(\sigma_{2,m})| \leq C_0 
	\left( \frac 2n +  \|\sigma_{1,m}-\sigma_{2,m}\|_{L^\infty(M\times \Gh)} \right ) c_0^2,
 	$$
 	with $C_0$ a constant of $M$.
 	Indeed, 		
considering $\tau_{i,m}\in \cA_0$, $i=1,2$, constructed as Lemma \ref{lem_symbtensor} (2) for the symbols $\sigma_{i,m}\in \cA_0$, $i=1,2$, and $\eta=1/m$, 
	we have
	\begin{align*}
		|\ell_\eps (\sigma_{1,m}-\sigma_{2,m})|
		 &\leq 
		|\ell_\eps (\sigma_{1,m} -\tau_{1,m})|
		+|\ell_\eps (\sigma_{2,m} -\tau_{2,m})|
+|\ell_\eps (\tau_{1,m} -\tau_{2,m})|\\
& \leq  C_0 \left( \frac 2m +  \|\sigma_1-\sigma_2\|_{L^\infty(M\times \Gh)} \right)\|\phi_\eps\|_{L^2(M)}^2 + 
 C\eps \|\phi_\eps\|_{L^2(M)}^2, 
	\end{align*}
	by Part (1) and Lemma \ref{lem_symbtensorquad}; 
		the observation follows from taking  the limit as $\eps=\eps_k$, $k\to \infty$.
\end{itemize}

Observation (2) implies the uniqueness of a limit for $(\ell(\sigma_m))_{m\in \bN}$ for a sequence $(\sigma_m)$  
in $\cA_0$ converging to $\sigma$ for $\|\cdot\|_{L^\infty(M\times \Gh)}$.
The existence is provided by Observation (1):
since $(\ell_0(\tau_\eta))_{\tau\in (0,1]}$ is bounded, there exists a sequence $(\eta_{n})_{n\in \bN}$ converging to $0$ as $n\to \infty$ such that $(\ell_0(\sigma_m))_{n\in \bN}$ converges, where $\sigma_m=\tau_{\eta_{n}}$. 
Moreover, the limit is  $\leq c_0^2\|\sigma\|_{L^\infty(M\times \Gh)}$.
This shows that the map $\ell_0$ extends naturally  to $C^*(M\times \Gh) $. The resulting map (for which we keep the same notation) is a linear functional on $C^*(M\times \Gh) $ that is positive and continuous with operator norm $\|\ell_0\|_{\sL(C^*(M\times \Gh) ,\bC)}\leq C_0  c_0^2$.
By \cite[Section 2.1]{Dixmier},
 it remains to show $\|\ell_0\|_{\sL( C^*(M\times \Gh),\bC)}=c_0^2$ 
to conclude the proof, and we already know
\begin{equation}
\label{eq_normell0}
	\|\ell_0\|_{\sL(C^*(M\times \Gh),\bC)} = \lim_{t\to 0} \ell_0(\sigma_{\kappa,t}),
\end{equation}
where $\sigma_{\kappa,t}$, $t>0$, is an approximate identity as in Proposition \ref{prop_algcA} (to fix the ideas). 
\end{proof}

\subsubsection{The von Neumann algebra $L^\infty(M\times \Gh)$}
\label{subsubsec_LinftyMGh}
From the group case recalled in Section \ref{subsubsec_vNG}
and the proof of Proposition \ref{prop_algcA} (1), 
routine arguments show that the von Neumann algebra generated by the $C^*$-algebra $C^*(M\times \Gh) $ is the space
$L^\infty (M\times \Gh)$
of measurable fields of operators that are bounded, that is, 
of measurable fields of operators
 $\sigma=\{\sigma(\dot x,\pi)\in \sL(\cH_\pi): (\dot x,\pi)\in \Gh\}$ such that
 $$
 \exists C>0\qquad \|\sigma(\dot x,\pi)\|_{\sL(\cH_{\pi})} \leq C 
 \ \mbox{for} \ d\dot xd\mu(\pi)\mbox{-almost all} \ (\dot x,\pi) \in M\times \Gh.
 $$
 The smallest of such constant $C>0$ is the norm 
 $$
 \|\sigma\|_{L^{\infty}(M\times \Gh)}=\sup_{(\dot x,\pi)\in M\times \Gh }\|\sigma(\dot x,\pi)\|_{\sL(\cH_{\pi})},
 $$
 with $\sup$ understood as an $d\dot xd\mu(\pi)$-essential supremum of $\sigma$ in $L^\infty (M\times \Gh)$. 

Naturally, $L^\infty (M\times \Gh)$ contains $C^*(M\times \Gh) $, but also  many other important symbols, such as the ones corresponding to symbols in the von Neumann but not $C^*$ algebra of the group $G$, for instance the identity symbol:
$$
\id :=\{\id_{\cH_\pi}, \pi\in \Gh\} \in L^\infty (\Gh), 
\qquad \id :=\{\id_{\cH_\pi}, (x,\pi)\in M\times \Gh\} \in L^\infty (M\times \Gh).
$$

We check readily that the subspace $C^*(M\times \Gh) \oplus \bC \id$ of $L^\infty(M\times \Gh)$ is a $C^*$-algebra for the $\|\cdot\|_{L^\infty(M\times \Gh)}$ norm, with unit $\id$ and  enveloping von Neumann algebra  $L^\infty(M\times \Gh)$.
	 We can now conclude the proof of  Proposition \ref{prop_scL} (3). 
	 
\begin{proof}[End of the proof of Proposition \ref{prop_scL} (3)]
We set for any $\eps\in (0,1]$, $\sigma_0\in \cA_0$, $\sigma_1\in C^*(M\times \Gh)$, and $a_0\in \bC$,
$$
	\tilde \ell_\eps (\sigma_0+a_0 \id) := \ell_\eps (\sigma_0) +a_0 \|\phi_\eps \|_{L^2(M)}^2\qquad\mbox{and}\qquad 
	\tilde \ell_0(\sigma_1+a_0 \id):=\ell_0(\sigma_1) +a_0 c_0^2.
	$$
	This defines a continuous linear functional $\tilde \ell_0$ on the Banach space $(C^*(M\times \Gh)\oplus \bC \id, \|\cdot\|_{L^\infty (M\times \Gh)})$, and   linear functionals 
	$\ell_\eps$ on the  vector space $\cA_0\oplus \bC \id$. They  satisfy for any $\sigma_0\in \cA_0$, $a_0\in \bC$, 
 $$
 \exists \lim_{k\to \infty} \tilde \ell_{\eps_k}(\sigma_0+a_0\id ) =\ell_0(\sigma_0) +a_0 c_0^2=\tilde \ell_0(\sigma_0+a_0\id).
 $$
and
	\begin{align*}
	\tilde \ell_\eps ((\sigma_0+a_0\id)(\sigma_0+a_0\id)^*) 
	&= ( \Op^{\eps} (\sigma_0 \sigma_0^*+a_0 \sigma^* +\bar a_0 \sigma) \phi_\eps ,\phi_\eps)_{L^2(M)} + |a_0|^2 \|\phi_\eps\|_{L^2(M)}^2	\\
	&=\|(\Op^{\eps} (\sigma_0) +a_0 \id )^*  \phi_\eps \|^2_{L^2(M)} +O(\eps),
	\end{align*}
having proceeded as above for $\ell_\eps$. Hence, $\tilde \ell_0 (\sigma\sigma^*) \geq 0$ 
for any $\sigma\in \cA_0\oplus \bC \id$ and then by density for any $\sigma \in C^*(M\times \Gh)\oplus\bC \id$.
In other words, $\tilde \ell_0$ is a continuous positive linear functional on $C^*(M\times \Gh)\oplus\bC \id$. 
We keep the same notation $\tilde \ell_0$ for the weakly continuous linear functional on the corresponding enveloping von Neumann algebra $L^\infty (M\times \Gh)$.
As any approximate identity of $C^*(M\times \Gh)$  converges weakly to $\id$ in $L^\infty (M\times \Gh)$,  we have
$
\lim_{t\to 0}\tilde \ell_0(\sigma_{\kappa,t}) = \tilde \ell_0 (\id) $
 for an approximate identity
 $\sigma_{\kappa,t}$, $t>0$, from Proposition \ref{prop_algcA}.
 We see $\tilde \ell_0 (\id) =c_0^2$.
As $\ell_0$ and $\tilde \ell_0$ coincide on  $C^*(M\times \Gh)$,
we have $\ell_0(\sigma_{\kappa,t})=\tilde \ell_0(\sigma_{\kappa,t})$.
We can now conclude the proof with \eqref{eq_normell0}.
\end{proof}

\subsection{The dual of $C^*(M\times \Gh) $ in terms of operator-valued measures}
\label{subsec_dualC*MGh}
Proposition \ref{prop_scL}  leads us to seek a better description for the positive linear functionals on $C^*(M\times \Gh) $. 
This will be  given in terms of operator-valued measures.

\subsubsection{Operator-valued measures} 
Let us recall the notion of operator-valued measure as defined in  \cite[Section 5]{FFPisa} and \cite[Section 2.6]{FFJST}:
\begin{definition}
\label{def_gammaGamma}
	Let $Z$ be a complete separable metric space, 
	and let $\xi\mapsto {\cH}_\xi$ a measurable field of complex Hilbert spaces of $Z$.
Denote by 
	$ \widetilde{\cM}_{ov}(Z,({\cH}_\xi)_{\xi\in Z})$
	 the set of pairs $(\gamma,\Gamma)$ where $\gamma$ is a positive Radon measure on~$Z$ 
	and $\Gamma=\{\Gamma(\xi)\in {\mathcal L}({\cH}_\xi):\xi \in Z\}$ is a measurable field of trace-class operators
such that
\begin{equation}
\label{eqdef_cMnorm}
\int_Z{\rm Tr}_{{\cH}_\xi} |\Gamma(\xi)|d\gamma(\xi)
<\infty.
\end{equation}

Two pairs $(\gamma,\Gamma)$ and $(\gamma',\Gamma')$ 
in $\widetilde {\cM}_{ov}(Z,({\cH}_\xi)_{\xi\in Z})$
are \emph{equivalent} when there exists a measurable function $f:Z\to \mathbb C\setminus\{0\}$ such that 
$$d\gamma'(\xi) =f(\xi)  d\gamma(\xi)\;\;{\rm  and} \;\;\Gamma'(\xi)=\frac 1 {f(\xi)} \Gamma(\xi)$$ for $\gamma$-almost every $\xi\in Z$.
The equivalence class of $(\gamma,\Gamma)$ is denoted by $\Gamma d \gamma$,
and the resulting quotient set is 
 denoted by ${\cM}_{ov}(Z,({\cH}_\xi)_{\xi\in Z})$.

A pair $(\gamma,\Gamma)$ 
in $ \widetilde {\cM}_{ov}(Z,({\cH}_\xi)_{\xi\in Z})$
 is \emph{positive} when 
$\Gamma(\xi)\geq 0$ for $\gamma$-almost all $\xi\in Z$.
In  this case, we may write  $(\gamma,\Gamma)\in  \widetilde {\cM}_{ov}^+(Z,({\cH}_\xi)_{\xi\in Z})$, 
and $\Gamma d\gamma \geq 0$ for $\Gamma d\gamma \in {\cM}_{ov}^+(Z,({\cH}_\xi)_{\xi\in Z})$.
\end{definition}

The quantity in \eqref{eqdef_cMnorm}  is constant on the equivalence class of $(\gamma,\Gamma)$
and is denoted by 
$$
\|\Gamma d \gamma\|_{\cM}.
$$
It is a norm on ${\cM}_{ov}(Z,({\cH}_\xi)_{\xi\in Z})$, which is then a Banach space. 

\subsubsection{The dual of $C^*(M\times \Gh) $}
Proceeding as in \cite[Section 5]{FFPisa} and \cite[Section 3.1]{FFJST}, we can identify the spectrum of $C^*(M\times \Gh) $ with $M\times \Gh$ and  the states of the $C^*$-algebra $C^*(M\times \Gh) $ with operator-valued measures. 
Following \cite[Section 3.1]{FFJST} and \cite[Section 5]{FFPisa}, we obtain the following description of the states of $C^*(M\times \Gh) $ as operator-valued measures:
for any pair $(\gamma,\Gamma) \in \widetilde {\cM}_{ov}(M\times \Gh)$, the linear functional $\ell_{\gamma, \Gamma}$ defined on $C^*(M\times \Gh) $ via 
$$
\ell_{\gamma, \Gamma}(\sigma) := \iint_{M\times \Gh} {\rm Tr}\left(\sigma(x,\pi) \Gamma(x, \pi)\right) d \gamma(x, \pi), \quad \sigma\in C^*(M\times \Gh),
$$
is continuous. Furthermore, it is independent of the equivalence class of $(\gamma,\Gamma)$, and  any continuous linear functional on $C^*(M\times \Gh) $ is of this form and is uniquely determined by a class $\Gamma d\gamma$.  In other words, the mapping 
$$
\left\{\begin{array}{rcl}
 {\cM}_{ov}(M\times \Gh) &\longrightarrow & (C^*(M\times \Gh))^* \\
\Gamma d\gamma &\longmapsto & \ell_{\gamma, \Gamma}=\ell_{\Gamma d \gamma}
\end{array}\right.
$$
is an isomorphism of Banach  spaces.
The states of $C^*(M\times \Gh) $ are then the linear functional $\ell_{\gamma, \Gamma}$  of norm one
with  $\Gamma d\gamma \in {\cM}^+_{ov}(M\times \Gh)$.

\subsubsection{Semi-classical measures}
Theorem \ref{thm_scL} in the introduction follows readily from  Proposition \ref{prop_scL} and the description above. 
In Theorem \ref{thm_scL}, 
$\Gamma d\gamma \in {\cM}^+_{ov}(M\times \Gh)$ is called a \emph{semi-classical measure} 
of the family $(\phi_\eps)_{\eps\in (0,1]}$ 
at scale $\eps$, or the semi-classical measure
for the  sequence $(\eps_k)_{k\in \bN}$. 
We extend this vocabulary to the  context of  a bounded sequence of functions $(\phi_j)_{j\in \bN}$ in $L^2(M)$ and a map $j\mapsto \eps_j$ valued in $(0,1]$ satisfying $\lim_{j\to \infty} \eps_j=0$.

We observe that under certain hypotheses, 
the semi-classical limits of a bounded family $(\phi_\eps)$ in $L^2(M)$ yield the accumulation points of the family of measures $|\phi_\eps(\dot x)|^2d\dot x$ for the weak-star topology.

\begin{definition}
	Let $(\phi_\eps)_{\eps\in (0,1]}$	be a bounded family in $L^2(M)$
	and let $\cR$ be a positive Rockland operator on $G$.

The family $(\phi_\eps)$
 is \emph{uniformly $\eps$-oscillating} with respect to $\cR_M$ when 
$$
\limsup_{\eps\to 0}\| 1_{\{\eps^\nu \cR_M \geq R\}} \phi_\eps\|_{L^2(M)} \longrightarrow_{R\to \infty} 0
$$
If in addition, 
$$
\limsup_{\eps\to 0}\| 1_{\{\eps^\nu \cR_M \leq \delta\}} \phi_\eps\|_{L^2(M)} \longrightarrow_{\delta\to 0} 0
$$
it is \emph{uniformly strictly $\eps$-oscillating}.
\end{definition}
We extend this vocabulary to the  context of  a bounded family of functions $(\phi_j)_{j\in J}$ in $L^2(M)$ and a map $j\mapsto \eps_j$ satisfying $\lim_{j\to \infty} \eps_j=0$. 

It is a routine exercise 
to check that if a bounded family $(\phi_\eps)_{\eps\in (0,1]}$ of $L^2(M)$ satisfies 
$$
\exists s>0, \qquad
\sup_{\eps \in (0,1]}\|(\eps^\nu \cR_M)^s \phi_\eps\|_{L^2(M)}<\infty.
$$
then it is uniformly $\eps$-oscillating for $\cR$, while if
$$
\exists s>0, \qquad
\sup_{\eps \in (0,1]}\|(\eps^\nu \cR_M)^s \phi_\eps\|_{L^2(M)}+
\|(\eps^\nu \cR_M)^{-s} \phi_\eps\|_{L^2(M)}<\infty.
$$
then it is uniformly strictly $\eps$-oscillating. 

Proceeding as in  \cite[Proposition 4.1]{FFJST}, we obtain:
\begin{proposition}
\label{prop_osc}
	Let $(\phi_\eps)_{\eps\in (0,1]}$ be a bounded family in $L^2(M)$ that  is uniformly $\eps$-oscillating for $\cR$.
Let $\Gamma d\gamma$ be a semi-classical measure, and denote by $(\eps_k)$ the corresponding subsequence. 
Then for any $f\in \cD(M)$, we have
$$
 (f\phi_{\eps_k} ,\phi_{\eps_k})_{L^2(M)} \longrightarrow_{k\to \infty} \iint_{M\times \Gh} f(\dot x)  {\rm Tr}\left( \Gamma(\dot x, \pi)\right) d \gamma(\dot x, \pi).
 $$
If $(\phi_\eps)$ is in addition uniformly strictly $\eps$-oscillating, then the semi-classical measure does not charge the trivial representation $1_{\Gh}$ in the sense that 
$$
\gamma (M\times \{1_{\Gh}\}) =0.
$$
\end{proposition}

The proof is omitted due to its similarity to the one of \cite[Proposition 4.1]{FFJST}: the only difference is that the symbols depend on $x\in G$ in \cite{FFJST} instead of  $\dot x\in M$ here, while the difficulty is the analysis on $\Gh$.  
Consequently, with the notation of the statement, 
the weak-star limit of $|\phi_{\eps_k}(\dot x)|^2 d\dot x$ is equal to $\int_{\Gh}  {\rm Tr}\left( \Gamma(x, \pi)\right) d \gamma(\dot x, \pi)$.
It admits a further decomposition, see Remark \ref{rem_decw*lim}.   

 \subsection{The  decomposition $\Gh= \Gh_\infty \sqcup \Gh_1$}
 \label{subsec_LinftyMGh}
 
The unitary dual of any nilpotent Lie group may  be described as the disjoint union 
$$
\Gh = \Gh_\infty \sqcup \Gh_1,
$$
of the class of the infinite dimensional representations parametrised $\Gh_\infty$
with the class of the finite dimensional representations $\Gh_1$.
By the orbit method, the finite dimensional representations are of dimension one
and  may be identified with a closed subset of $\Gh$ given by  characters:
$$
\Gh_1:= \{\pi^\omega \ \mbox{given by} \ \chi_\omega(v) =e^{i \omega(v)} \ : \ \omega \in [\fg,\fg]^\perp  \},
$$
with $[\fg,\fg]^\perp $ denoting the subspace of linear forms $\omega\in \fg^*$ such that $\omega=0$ on $[\fg,\fg]$.
To fix the idea, we  identify 
$$
\Gh_1 \sim [\fg,\fg]^\perp \sim \fv^*  ,
$$ 
with the dual $\fv^*$ of a subspace $\fv$ of $\fg$; as we have already fixed a basis of $\fg$, we can choose $\fv$ to be the orthogonal complement for the corresponding scalar on $\fg$.

Any linear functional $\ell=\ell_{\Gamma d\gamma} \in (C^*(M\times \Gh))^{*}$ admits a  unique extension  to a weakly continuous linear functional on the von Neumann algebra $L^\infty(M\times \Gh)$ enveloping $C^*(M\times \Gh) $.
 The decomposition $\Gh= \Gh_\infty \sqcup \Gh_1$ will turn out to be useful to decompose these extended linear functionals. 
 Indeed, as the representation in $\Gh_1$ are of dimension one,  $1_{M \times \Gh_1} \Gamma d\gamma $ is a scalar valued measure. 
Hence, any operator valued measure $\Gamma d\gamma \in \cM_{ov}(M\times \Gh)$ may be decomposed into the sum of one scalar valued measure and one operator-valued measure:
\begin{equation}
	\label{eq_GdgdecGh}
	\Gamma d\gamma = 1_{M \times \Gh_1}  \Gamma d\gamma 
+
1_{M \times \Gh_\infty}
\Gamma d\gamma .
\end{equation}

Let us understand this decomposition on linear functionals $\ell$.
First, we will need  the following observation:
\begin{ex}
\label{ex_1MGh1infty}
	The symbols $1_{M\times \Gh_1}$ and $1_{M\times \Gh_\infty}$ are in $L^\infty(M\times\Gh)$.
	Therefore, for any $\sigma\in C^*(M\times \Gh) $, the symbols $\sigma 1_{M\times \Gh_1}$ and $\sigma 1_{M\times \Gh_\infty}$ are in 
$L^\infty(M\times \Gh)$.
\end{ex}
This example together with the natural extension of any functional $\ell=\ell_{\Gamma d\gamma}\in (C^*(M\times \Gh))^*$ to $L^\infty(M\times \Gh)$ shows that the following decomposition makes sense: 
$$
\forall \sigma\in C^*(M\times \Gh) \qquad 
\ell (\sigma) = \ell(\sigma 1_{M \times \Gh_1}) +
\ell(\sigma 1_{M \times \Gh_\infty}).
$$

The operator-valued measure $1_{M \times \Gh_1}  \Gamma d\gamma $ is a scalar measure on $M\times \Gh_1 \sim M\times \fv^*$.
We may assume that $\Gamma=1$ on 
$M \times \Gh_1$, and by convention we do. 
The following lemma shows that a symbol $\sigma 1_{M\times \Gh_1}$, $\sigma\in C^*(M\times \Gh) $, may be identified with an element of  $C_0(M\times \fv^*)$; 
this provides a simpler description of the scalar measure $1_{M \times \Gh_1}\gamma$ that  will be useful in Proposition \ref{prop_meascomcR} (2):

\begin{lemma}
\label{lem_C0}
\begin{enumerate}
	\item Any $\sigma\in \cA_0$, $\sigma 1_{M\times \Gh_1}$,  identified with the restriction of $\sigma$ to $M\times \Gh_1$, coincides with the  element in $C^\infty(M;\cS(\fv^*))$
$$
\dot x \longmapsto (\omega\mapsto \sigma (\dot x ,\pi^\omega)) .
$$
\item 
Moreover, 
the closure for $L^\infty(M\times \Gh)$ of the algebra of symbols $\sigma 1_{M\times \Gh_1}$, 
	$\sigma\in \cA_0$, is equal to $\{\sigma 1_{M\times \Gh_1}, \sigma\in C^*(M\times \Gh) \}$ and is
	naturally isomorphic to the $C^*$-algebra $C_0(M\times \fv^*)$ of continuous functions on $M\times \fv^*$ vanishing at infinity. 
	\end{enumerate}
\end{lemma}

\begin{proof}
Before starting the proof, we observe that we can  choose the canonical basis $X_1,\ldots, X_n$ adapted to the gradation of $\fg$ and to the decomposition $\fg= \fv \oplus [\fg,\fg]$.
This allows us to define the Lebesgue measures $dV$ and $dZ$ on $\fv$ and $[\fg,\fg]$ such that $dVdZ$ is the exponential pull back of the Haar measure: 
$$
\forall f\in L^1(G)\qquad 
\int_G f(x) dx = \iint_{\fv\times [\fg,\fg]} f(\exp (V+Z)) dV dZ.
$$ 
The derived group $[G,G]$, whose Lie algebra is $[\fg,\fg]$, is therefore equipped with a compatible Haar measure:
$$
\int_{[G,G]} f(z) dz
=
\int_{[\fg,\fg]} f(\exp Z) dZ.
$$

Let $\sigma\in \cA_0$, and denote its kernel by $\kappa_{\dot x}(y)$.
Then 
\begin{equation}
\label{eq_pflem_C0}
	\sigma(\dot x,\pi^\omega) = \int_\fv \kappa_{\dot x} (\exp (V + Z)) \overline \chi_\omega(V) dV dZ
=
\cF_\fv (\int_{[\fg,\fg]} \kappa_{\dot x})(\omega),\qquad \dot x\in M, \ \omega\in \fv^*,
\end{equation}
where $\cF_\fv$ denotes the Euclidean Fourier transform on $\fv$, 
and $\int_{[G,G]}$ the integration on $[G,G]$. 
Since $\kappa\in C^\infty (M;\cS(G))$, Part (1) follows. 

The computation \eqref{eq_pflem_C0} also implies that 
given any $t\in C^\infty (M;\cS(\fv^*))$, we can find a symbol $\tau \in \cA_0$ such that $\tau 1_{M\times \Gh_1}$ coincides with $t$. 
Indeed, fixing a function $\chi\in \cD([\fg,\fg])$ satisfying $\int_{[\fg,\fg]}\chi(Z) dZ=1$,
 the kernel
$\kappa$ defined via
$$
\kappa_{\dot x} (\exp (V+Z)) = \chi (Z) \cF_\fv^{-1} t (\dot x ,\cdot), \qquad V\in \fv,\ Z\in [\fg,\fg],
$$
is in  $C^\infty (M;\cS(G))$, 
and by \eqref{eq_pflem_C0},  
its symbol $\sigma$ satisfies $\sigma(x,\pi^\omega) = t(\dot x,\omega)$ for all 
$(\dot x,\omega)\in M\times \Gh_1$.
Consequently, 
the $L^\infty(M\times \Gh)$ -closure  of  $\{\sigma 1_{M\times \Gh_1}, 
	\sigma\in \cA_0\}$, identifies naturally 
	with the closure of $C^\infty (M;\cS(\fv^*))$ for the supremum norm on $M\times \fv^*$, 
	and therefore identifies naturally with  $C_0(M\times \fv^*)$.

Part (1) implies that the $L^\infty(M\times \Gh)$ -closure  of  $\{\sigma 1_{M\times \Gh_1}, 
	\sigma\in C^*(M\times \Gh) \}$ also identifies with a $C^*$-sub-algebra of $C_0(M\times \fv^*)$.
It is straightforward to check that 
$$
\forall \sigma\in C^*(M\times \Gh) \qquad 
\|\sigma 1_{M\times \Gh_1}\|_{L^\infty (M\times \Gh)}
\leq \|\sigma\|_{L^\infty (M\times \Gh)}.
$$
So the $L^\infty(M\times \Gh)$ -closure  of  $\{\sigma 1_{M\times \Gh_1}, 
	\sigma\in \cA_0\}$, is a $C^*$-subalgebra of  $\{\sigma 1_{M\times \Gh_1}, \sigma\in C^*(M\times \Gh) \}$.
As the former identifies with the whole of $C_0(M\times \fv^*)$, they are both equal to $C_0(M\times \fv^*)$.
\end{proof}

\begin{remark}
\label{rem_decw*lim}
Let us consider the setting of Proposition \ref{prop_osc}:
let $\Gamma d\gamma$ be the semi-classical limit of a bounded family $(\phi_\eps)$ in $L^2(M)$ that is $\eps$-oscillating for a positive Rockland operator $\cR_M$ and a subsequence $\eps_j$.
 Then we have obtained a further decomposition of the weak-star limit 
\begin{align*}
	\lim_{k\to \infty} |\phi_{\eps_k}(\dot x)|^2 d\dot x
&= \int_{\Gh} \tr (\Gamma (\dot x,\pi)) d\gamma (\dot x,\pi)
\\&= \int_{\Gh_1}  d\gamma (\dot x,\pi)
+\int_{\Gh_\infty} \tr (\Gamma (\dot x,\pi)) d\gamma (\dot x,\pi).
\end{align*}
With different means (in fact, with the Euclidean micro-local analysis), this observation was already obtained in particular cases that essentially boil down to  Heisenberg nilmanidolds, see \cite{CdvHT,letrouit}.
\end{remark}

\section{Quantum limits for $\cR_M$}
\label{sec_QLR}

In this section, we study the semi-classical limit in our theory associated with a sequence of eigenfunctions of $\cR_M$, where $\cR$ is a (fixed) positive Rockland operator.
We obtain properties of localisation and invariance in Proposition \ref{prop_sclmR}
and Corollary \ref{cor_prop_sclmR}
 below, described more explicitly in the case of sub-Laplacians in Section \ref{subsubec_caseL}.
 The proof requires to analyse symbols and operator-valued measures commuting with $\widehat \cR$ in Section \ref{subsec_symbcomR}.

\subsection{Family of $\cR_M$-eigenfunctions}
\label{subsec_sclmR}

Let $\cR$ be a positive Rockland operator on $G$ and denote by $\cR_M$  the corresponding operator on $M$, as well as its  self-adjoint extension to $L^2(M)$. Recall that 
the spectral decompositions of $\cR$ and $\pi(\cR)$ for $\pi\in \Gh$ are denoted by 
 $E$ and $\pi(E)$, see Section \ref{subsec_cR}.
 Moreover,   $\pi^\omega(\cR)$ is the (scalar) eigenvalue of $\cR$ for $\chi_\omega$: $\cR \chi_\omega = \pi^\omega(\cR) \chi_\omega$.

We can already obtain properties of localisation of the semi-classical measures of a sequence of $\cR_M$-eigenfunctions with eigenfunctions going to infinity:
\begin{proposition}
\label{prop_sclmR}
	Let $(\phi_j)_{j\in \bN}$ be a sequence of eigenfunctions  for $\cR_M$:
 $$
 \cR_M \phi_j = \mu_j  \phi_j, \quad j=0,1,2,\ldots. 
 $$
 Assume $\mu_j \to \infty$ as $j\to \infty$. 
   Consider a semi-classical measure $\Gamma d\gamma$ of $(\phi_j)$ at scale $\mu_j^{-1/\nu}$ for the subsequence $(j_k)$ where $\nu$ is the degree of homogeneity of $\cR$.
 We have for $\gamma$-almost all $(\dot x,\pi)\in M\times \Gh$
 $$
 \Gamma (\dot x,\pi) = \pi(E_1) \Gamma (\dot x,\pi) \pi(E_1).
 $$
The decomposition \eqref{eq_GdgdecGh} of $\Gamma d\gamma$ according to $\Gh = \Gh_1 \sqcup \Gh_\infty$ 
   satisfies the following properties:
  \begin{enumerate}
  	\item   The scalar valued measure
 $1_{M\times \Gh_1}\gamma$ on $M\times \Gh_1$ is supported in 
 $$
 M\times \{\pi^\omega\in \Gh_1: \pi^\omega(\cR)=1\}\sim M\times \{\omega\in \fv^* : \cR \chi_\omega= \chi_\omega\}.
 $$  
\item 
For $\gamma$-almost all $(\dot x, \pi)\in M\times \Gh_\infty$, the operator
  $\Gamma(\dot x,\pi)$ maps the finite dimensional 1-eigenspace for $\pi(\cR)$ onto itself and is trivial anywhere else.
 \end{enumerate}
 \end{proposition}
 
 \begin{proof}
 The properties of the functional and semi-classical calculi imply for any $\sigma\in \cA_0$
 $$
 (\Op^{(\eps)} (\sigma \widehat \cR) \phi_j, \phi_j)
 =
 (\Op^{(\eps)} (\sigma) \eps^\nu \cR_M \phi_j, \phi_j)
 =
 \eps^\nu \mu_j (\Op^{(\eps)} (\sigma) \phi_j, \phi_j).
 $$
Similarly, since $\cR$ is self-adjoint, we have:
 \begin{align*}
 (\Op^{(\eps)} ( \widehat \cR\sigma) \phi_j, \phi_j)
 &=
 (\eps^\nu \cR \Op^{(\eps)} (\sigma) \phi_j, \phi_j) + O(\eps)
=
\eps^\nu  (\Op^{(\eps)} (\sigma) \phi_j, \cR  \phi_j) + O(\eps)
\\& =
 \eps^\nu \mu_j (\Op^{(\eps)} (\sigma) \phi_j, \phi_j) +O(\eps).
 \end{align*}
  We take $\eps_j = \mu_j^{-1/\nu}$ and $j=j_k$, and then the limit as $k\to \infty$ to obtain for any $\sigma\in \cA_0$:
 \begin{equation}
 \label{eq_int_pfprop_sclmR}
 	 \int_{M\times \Gh} \tr (\sigma \widehat \cR \Gamma) d\gamma
 =
   \int_{M\times \Gh} \tr (\sigma  \Gamma) d\gamma
 =
   \int_{M\times \Gh} \tr (\widehat \cR \sigma \Gamma ) d\gamma.
 \end{equation}
 
 As explained in Section \ref{subsec_LinftyMGh}, 
 any linear function $\ell=\ell_{\Gamma d\gamma}$ admits a natural unique extension to a continuous linear function on the von Neumann algebra $L^\infty(M\times \Gh)$.
 Hence, \eqref{eq_int_pfprop_sclmR} is also valid for any symbol $\sigma$ in $L^\infty(M\times \Gh)$, 
 in particular for  $\sigma 1_{M\times \Gh_1}$ and also $\sigma 1_{M\times \Gh_\infty} \widehat E_\lambda$ or $\widehat E_\lambda \sigma 1_{M\times \Gh_\infty} $ below.
 Consequently, \eqref{eq_int_pfprop_sclmR} also holds for $\sigma 1_{M\times \Gh_1}$, $\sigma\in \cA_0$, and 
      we obtain 
   $$
   \forall f\in C_0(M\times \fv^*)\qquad
   \qquad 
 \int_{M\times \fv^*} f(\dot x,\omega) \pi^\omega(\cR) d\gamma(\dot x,\pi^\omega)
 =
 \int_{M\times \fv^*} \sigma(\dot x,\pi^\omega)  d\gamma(\dot x,\pi^\omega).
   $$
   Part (1)  follows. 
  Equality   \eqref{eq_int_pfprop_sclmR}
 also holds for any $\sigma 1_{M\times \Gh_\infty} \widehat E_\lambda$ with $\lambda\in \bR$ fixed, see \eqref{eq_Elambda}.
 Since we have $\widehat E_\lambda \widehat \cR = \widehat E_\lambda \lambda$,  we obtain for any $\sigma\in \cA_0$:
  $$
 \int_{M\times \Gh_\infty} \tr (\sigma \widehat E_\lambda \lambda \Gamma) d\gamma
 =
\int_{M\times \Gh_\infty} \tr (\sigma \widehat E_\lambda \widehat \cR \Gamma) d\gamma
 =
   \int_{M\times \Gh_\infty} \tr ( \sigma \widehat E_\lambda  \Gamma) d\gamma.
    $$
 By uniqueness, we have
 $ \lambda \widehat E_\lambda \Gamma  
 = 
 \widehat E_\lambda \Gamma$, which means that the projection of the image of $\Gamma (x,\pi)$ 
 onto the (finite dimensional) $\lambda$-eigenspaces of $\pi(\cR)$
  is zero  except perhaps for $\lambda =1$.
 Similarly,   Equality   \eqref{eq_int_pfprop_sclmR}
 also holds for any $\widehat E_\lambda\sigma 1_{M\times \Gh_\infty} $ so
 $ \lambda \Gamma  \widehat E_\lambda 
 = 
  \Gamma \widehat E_\lambda$, which means that $\Gamma (x,\pi)=0$ on all the (finite dimensional) $\lambda$-eigenspaces of $\pi(\cR)$ except perhaps for $\lambda =1$. 
This concludes the proof of Part (2). 
     \end{proof}

\subsection{Invariance properties}  
  \label{subsubsec_motivationsymbcomR}

  The properties of the semi-classical calculus  can give us further information on the semi-classical measures, in particular properties of invariance. 
 Indeed, the property \eqref{eq_Opbracketvarphi}  implies  for any $\sigma\in \cA_0$  commuting with $\widehat \cR$,  
 $$
\forall \eps\in (0,1],\ j\in \bN\qquad
(\Op^{(\eps)}(\sE  \sigma)\phi_j,\phi_j) =   O(\eps),
$$
where $\sE := \sum_{[\alpha] = \upsilon_1} \Delta^{\alpha}\widehat \cR X_{\dot x}^{\alpha}.$
Taking $\eps=\eps_j$ and $j\to \infty$, Proposition \ref{prop_sclmR} implies the following property of invariance:
$$
 \iint_{M\times \Gh} {\rm Tr}\left(\sE \sigma\ \Gamma\right) d \gamma=0.
$$
The rest of this section is devoted to showing the following lemma stating
 that the equality just above also holds on $M\times \Gh_1$ and on $M\times \Gh_\infty$ separately:
\begin{corollary}
\label{cor_prop_sclmR}
We continue with the setting of Proposition \ref{prop_sclmR}.
We have 
$$
\forall f\in \cD(M\times \fv^*)\qquad\iint_{M\times \Gh_1} \sum_{[\alpha]=\upsilon_1}  \Delta^\alpha \widehat \cR (\chi_\omega)  \ X_{M}^\alpha f(\dot x,\omega)\  d \gamma(\dot x,\pi^\omega)=0,
$$
and 
$$
\forall \sigma\in \cA_0^{\widehat \cR}\qquad 
 \iint_{M\times \Gh_\infty} {\rm Tr}\left(\sE \sigma \ \Gamma\right) d \gamma=0.
$$
\end{corollary}
The proof of Corollary \ref{cor_prop_RM} will be given  in Section \ref{subsubsec_pfcor_prop_sclmR}. It will rely on the properties of symbols commuting with $\widehat \cR$ we now present.

\subsection{Symbols and operator-valued measures commuting with $\widehat \cR$}
\label{subsec_symbcomR}

\subsubsection{Spaces of symbols commuting with (the spectral decomposition of) $\widehat \cR$}

Here, we study the  spaces of symbols commuting with the spectral decompositions  $\widehat E$ of $\cR$, see Section \ref{subsec_cR}:
 \begin{definition}
We denote by 
$$
\cA_0^{\widehat \cR},\ (C^*(M\times \Gh))^{\widehat \cR}
\quad\mbox{and}\quad
L^\infty (M\times \Gh)^{\widehat \cR},
$$ the subsets of elements respectively in $\cA_0$, $C^*(M\times \Gh) $ and $L^\infty(M\times \Gh)$  that commute with $\widehat E(I)$ for every interval $I\subset\bR$.
\end{definition}

As may be expected, these spaces enjoy the following properties:

\begin{proposition}
\label{prop_symbcomcR}
\begin{enumerate}
\item A symbol $\sigma$ is in $\cA_0^{\widehat \cR}$ if and only if it commutes with $\widehat \cR$. 
The space $\cA_0^{\widehat \cR}$ is a subalgebra of $\cA_0$ that contains all the symbols given by  $a(\dot x) \psi (\pi(\cR))$, where $a\in C^\infty(M)$ and $\psi\in \cS(\bR)$. 
In particular, it contains the approximate identity of $C^*(M\times \Gh) $ given by $\psi_j(\widehat \cR)$, $j\in \bN$,
where  $\psi_j\in \cD(\bR)$ is valued in $[0,1]$ and satisfies $\psi_j(\lambda)=1$ on $\{|\lambda| \geq j\}$. 
\item The space $(C^*(M\times \Gh))^{\widehat \cR}$ is a $C^*$-subalgebra of $C^*(M\times \Gh) $. 
It  is the closure of $\cA_0^{\widehat \cR}$ for the $L^\infty(M\times \Gh)$-norm.
\item The space $L^\infty (M\times \Gh)^{\widehat \cR}$ is a von Neumann subalgebra of $L^\infty (M\times \Gh)$. It is the  von Neumann algebra envoloping $(C^*(M\times \Gh))^{\widehat \cR}$.
\end{enumerate}
\end{proposition}

\begin{proof}
If $\sigma\in \cA_0^{\widehat \cR}$, then $\sigma(\dot x,\pi)$ commutes with $\pi(\cR) = \int_{\bR} \lambda d \pi(E_\lambda)$ for all $(\dot x,\pi)\in M\times \Gh$. 
The converse holds  for any $\pi\in \Gh\setminus \{1_{\Gh}\}$ since the spectrum of $\pi(\cR)$ is a discrete unbounded subset of $(0,\infty)$ and its eigenspaces are finite dimensional,  see Section \ref{subsec_cR}.
The rest of Part (1) follows from routine checks and Hulanicki's theorem (Theorem \ref{thm_hula}).

As $\cA_0$ is dense in $C^*(M\times \Gh) $,  
$(C^*(M\times \Gh))^{\widehat \cR}$ is the closure of 
$\cA_0^{\widehat \cR}$ and it is a $C^*$-subalgebra of $C^*(M\times \Gh) $. This proves Part (2).
As the $C^*$-algebra $C^*(M\times \Gh) $ generates the von Neumann algebra $L^\infty (M\times \Gh)$ (for instance, in the sense of taking the closure for the Strong Operator Topology), 
the $C^*$-algebra $(C^*(M\times \Gh))^{\widehat \cR}$ generates the von Neumann algebra $L^\infty (M\times \Gh)^{\widehat \cR}$ which is a von Neumann subalgebra of $L^\infty (M\times \Gh)$. This proves Part (3). 
\end{proof}

 \subsubsection{Space of operator valued measures commuting with $\widehat E$}
We say that the pair $(\gamma,\Gamma) \in \widetilde \cM_{ov}(M\times \Gh)$ commutes with $\widehat E$ when 
 for any interval $I\subset \bR$, we have 
 $ \Gamma (\dot x,\pi) \pi(E(I))= \pi(E(I)) \Gamma (\dot x,\pi)$ for $\gamma$-almost all $(\dot x,\pi)\in M\times \Gh$. 
  We check readily that if  $(\gamma,\Gamma) \in \widetilde \cM_{ov}(M\times \Gh)$ commutes with $\widehat E$, then so does any pair equivalent to it. 
We then say that the operator valued measure $\Gamma d\gamma$ commutes with $\widehat E$. 
\begin{definition}
	We denote the space of $\Gamma d\gamma$ commuting with $\widehat E$ by 
$\cM_{ov}(M\times \Gh)^{\widehat \cR}.$ 
\end{definition} 

 As may be expected, this space may identified with the dual of $(C^*(M\times \Gh))^{\widehat \cR}$ amongst other properties: 
\begin{proposition}
\label{prop_meascomcR}
\begin{enumerate}
\item The space $\cM_{ov}(M\times \Gh)^{\widehat \cR}$ is a closed subspace of $\cM_{ov}(M\times \Gh)$.
Its  image under $\Gamma d\gamma \mapsto \ell_{\Gamma d\gamma}$ is the dual of $(C^*(M\times \Gh))^{\widehat \cR}$.
\item We have (with the closure for the $L^\infty(M\times \Gh)$-norm)
\begin{align*}
\overline{\{\sigma 1_{M\times \Gh_1}, \sigma\in \cA_0^{\widehat \cR}\}}
&=
\{\sigma 1_{M\times \Gh_1}, \sigma\in (C^*(M\times \Gh))^{\widehat \cR}\}
\\&=
\{\sigma 1_{M\times \Gh_1}, \sigma\in C^*(M\times \Gh) \}
\sim C_0(M\times \fv^*).	
\end{align*}
\end{enumerate}
\end{proposition}

The rest of this paragraph is devoted to the proof of Proposition \ref{prop_meascomcR}.
In order to prove Part 1, we will need the following result of functional analysis:
\begin{lemma}
\label{lem_FA}
Let $R$ be a self-adjoint operator on a Hilbert space $\cH$.
We assume that its spectrum $\sp(R)$ is discrete with accumulation point at infinity and that every eigenspace is
  finite dimensional.
  Denoting by $F_\zeta$ the  projection onto the $\zeta$-eigenspace, given any  operator $A\in \sL(\cH)$, we may define 
  $$
  A^R= \sum_{\zeta\in \sp(R)} F_\zeta A F_\zeta.
  $$
  Then $A^R$ is a bounded operator on $\cH$ that commutes with any $F_\zeta$, $\zeta\in \sp(R)$. Moreover, $\|A^R\|_{\sL(\cH)}\leq \|A\|_{\sL(\cH)}$ and if $A$ is in addition trace-class then so is $A^R$ with $\tr |A^R|\leq \tr |A|$.
\end{lemma}
\begin{proof}[Proof of Lemma \ref{lem_FA}]
Let $A\in \sL(\cH)$. Then 
$$
\|A^R u\|_\cH^2 = \sum_{\zeta\in \sp(R)} (F_\zeta A F_\zeta u,  A F_\zeta u)_\cH\leq \|A\|_{\sL(\cH_\pi)}^2\sum_{\zeta\in \sp(R)} \|F_\zeta u\|_\cH^2
=  \|A\|_{\sL(\cH_\pi)}^2 \| u\|_\cH^2,$$
since the $F_\zeta$'s are orthogonal projections satisfying $\sum_{\zeta\in \sp(R)} F_\zeta =\id_\cH$. 
This proves 	$\|A^R\|_{\sL(\cH)}\leq \|A\|_{\sL(\cH)}$. 
If $A$ is trace class then for any $B\in \sL(\cH)$,
$$
|\tr (A^R B)| = |\tr (A B^R) | \leq \tr |A| \|B^R\|_{\sL(\cH_\pi)} \leq \tr |A| \|B\|_{\sL(\cH_\pi)}
$$
showing $\tr |A^R|\leq \tr |A|$.
\end{proof}

\begin{proof}[Proof of Proposition \ref{prop_meascomcR} (1)]
With the spectral properties of $\pi(\cR)$  recalled in Section \ref{subsec_cR},  given  $(\Gamma,\gamma) \in \widetilde {\cM}_{ov}(M\times \Gh)$, we can define
the measurable field $ \Gamma^{(\widehat \cR)}$ as in Lemma \ref{lem_FA} via
$$
\Gamma^{(\widehat \cR)} (\dot x,\pi):= \sum_{\zeta\in \sp(\pi(\cR))} \pi(E_\zeta) \Gamma (\dot x,\pi) \pi(E_\zeta),
\qquad (\dot x,\pi)\in M\times \Gh;
$$
Clearly, $\Gamma^{(\widehat \cR)}$ is trace-class with 
$\tr |\Gamma^{(\widehat \cR)} (\dot x,\pi)|
\leq \tr |\Gamma (\dot x,\pi)|$
and commutes with $\pi(E_\zeta)$, $\zeta\in \sp(\pi(\cR))$.
This implies that  $( \Gamma^{(\widehat \cR)},\gamma)\in \widetilde \cM_{ov}(M\times \Gh)$ commutes with $\widehat E$. 
This construction passes through the quotient, defining a  map $\Gamma d\gamma  \mapsto \Gamma^{(\widehat \cR)} d\gamma$ on 
$\cM_{ov}(M\times \Gh)$.
We check readily that this map is linear, 1-Lipschitz for $\|\cdot\|_{\cM}$ and that its image is  $\cM_{ov}(M\times \Gh)^{\widehat \cR}$ where it is the identity; moreover, this is a continuous projection onto $\cM_{ov}(M\times \Gh)^{\widehat \cR}$.
Hence, $\cM_{ov}(M\times \Gh)^{\widehat \cR}$ is a closed subspace of $\cM_{ov}(M\times \Gh)$.

Let $\ell: (C^*(M\times \Gh))^{\widehat \cR} \to \bC$ be a continuous linear functional.
By the Hahn-Banach theorem, $\ell$ extends continuously to an element $\widetilde\ell$ of $(C^*(M\times \Gh))^*$. Let $\Gamma d\gamma \in \cM_{ov}(M\times \Gh)$ be such that $\widetilde\ell =\ell_{\Gamma d \gamma}$.
If $\sigma \in (C^*(M\times \Gh))^{\widehat \cR}$ then using the resolution $\sum_{\zeta \in \sp(\pi(\cR))} \pi(E_\zeta) = \id_{\cH_\pi}$ of the identity operator $\id_{\cH_\pi}$
and $\pi(E_\zeta)^2=\pi(E_\zeta)$, we have:
\begin{align*}
	\ell(\sigma)&= \widetilde\ell(\sigma) = \iint_{M\times \Gh} {\rm Tr}\left(\sigma(x,\pi) \Gamma(x, \pi)\right) d \gamma(x, \pi) \\
	&= \iint_{M\times \Gh} \sum_{\zeta \in \sp(\pi(\cR))} {\rm Tr}\left(\pi(E_\zeta)\sigma(x,\pi)\pi(E_\zeta) \Gamma(x, \pi)\right) d \gamma(x, \pi)	\\
	&= \iint_{M\times \Gh}  {\rm Tr}\left(\sigma(x,\pi) \Gamma^{(\widehat \cR)}(x, \pi)\right) d \gamma(x, \pi).
\end{align*}
In other words,  $\widetilde\ell =\ell_{\Gamma^{(\widehat \cR)} d \gamma}$ on $(C^*(M\times \Gh))^{\widehat \cR}$.
Consequently $\Gamma d\gamma \mapsto \ell_{\Gamma d\gamma}\big|_{(C^*(M\times \Gh))^{\widehat \cR}}$ maps 
$\cM_{ov}(M\times \Gh)^{\widehat \cR}$ into $((C^*(M\times \Gh))^{\widehat \cR})^*$.
 This map is continuous and linear. It remains to show that it is injective. 
 
 Let $\Gamma d\gamma \in \cM_{ov}(M\times \Gh)^{\widehat \cR}$ such that $\ell_{\Gamma d\gamma}=0$ on $(C^*(M\times \Gh))^{\widehat \cR}$.
 The natural extension of  $\ell_{\Gamma d\gamma}$ to $L^\infty(M\times \Gh)$ (for which we keep the same notation) will vanish on  $L^\infty(M\times \Gh)^{\widehat \cR}$.
  Note that if $\sigma\in C^*(M\times \Gh) $ and $\zeta\in \bR$ then $\widehat E_\zeta\sigma \widehat E_\zeta\in L^\infty(M\times \Gh)^{\widehat \cR}$ (see \eqref{eq_Elambda}),
 so  
$\ell (\widehat E_\zeta\sigma \widehat E_\zeta) = 0$.
With a similar computation as above, this shows that $\widehat E_\zeta \Gamma \widehat E_\zeta d\gamma =0$.
But this implies  $\Gamma^{(\widehat \cR)} d\gamma =0=\Gamma d\gamma $.
By linearity, this shows the injectivity. 
 \end{proof} 

\begin{proof}[Proof of Proposition \ref{prop_meascomcR} (2)]
By Lemma \ref{lem_C0} and its proof, we have
\begin{align*}
\overline{\{\sigma 1_{M\times \Gh_1}, \sigma\in \cA_0^{\widehat \cR}\}}
&=
\{\sigma 1_{M\times \Gh_1}, \sigma\in (C^*(M\times \Gh))^{\widehat \cR}\}
\\&\subseteq
\{\sigma 1_{M\times \Gh_1}, \sigma\in C^*(M\times \Gh) \}
\sim C_0(M\times \fv^*).	
\end{align*}
Let $\ell:\{\sigma 1_{M\times \Gh_1}, \sigma\in C^*(M\times \Gh) \}\to \bC$ be a continuous linear functional. It may be identified with a Radon measure $\gamma$ on $M\times \Gh_1 \sim M\times \fv^*$ via
$$
\ell (\sigma 1_{M\times \Gh_1}) = \int_{M\times \Gh_1} \sigma(\dot x, \pi^\omega) d\gamma(x,\pi^\omega).
$$
It is naturally extended into the operator-valued measure $\Gamma d\gamma  $ with $1_{M\times \Gh_1} \Gamma d\gamma = \gamma$ 
and $1_{M\times \Gh_\infty} \Gamma d\gamma =0$.
In other words, 
$$
\Gamma(\dot x,\pi)d\gamma (\dot x,\pi)
=\left\{\begin{array}{ll}
0&\mbox{if} \ \pi\in \Gh_\infty\\
	d\gamma (\dot x,\pi^\omega) & \mbox{if} \ \pi=\pi^\omega \in \Gh_1
	\end{array}\right.
	$$
As $\pi^\omega(E(I))$ and $\Gamma (\dot x, \pi^\omega)$ are scalar, they commute; this shows that  $\Gamma d\gamma$ commutes with $\widehat E$.  

If $\ell$ vanishes on $\{\sigma 1_{M\times \Gh_1}, \sigma\in (C^*(M\times \Gh))^{\widehat \cR}\}$ then $\ell_{\Gamma d\gamma}=0$ on $C^*(M\times \Gh)^{\widehat \cR}$ and $\Gamma d\gamma=0$ by Proposition \ref{prop_symbcomcR} (1), so $\gamma=0$ and $\ell=0$.

We have shown that if $\ell$ is a continuous linear functional on $
\{\sigma 1_{M\times \Gh_1}, \sigma\in C^*(M\times \Gh) \}$ vanishing on the subspace $\{\sigma 1_{M\times \Gh_1}, \sigma\in \cA_0^{\widehat \cR}\}$ then $\ell=0$. 
By the Hahn-Banach theorem, the closure of $\{\sigma 1_{M\times \Gh_1}, \sigma\in \cA_0^{\widehat \cR}\}$ is equal to $
\{\sigma 1_{M\times \Gh_1}, \sigma\in C^*(M\times \Gh) \}$.
\end{proof}  

\subsubsection{Proof of Corollary \ref{cor_prop_sclmR}}
\label{subsubsec_pfcor_prop_sclmR} 

In Section \ref{subsubsec_motivationsymbcomR}, 
 we obtained:
\begin{equation}
	\label{eq_pfcor_prop_sclmR}
	\forall \sigma\in \cA_0^{\widehat \cR}\qquad
\ell_{\Gamma d\gamma}(\sE \sigma)=0.
\end{equation}
By Proposition \ref{prop_meascomcR} (1), 
the extension of $\ell_{\Gamma d\gamma}|\cA_0^{\widehat \cR}$ to $L^\infty (M\times \Gh)^{\widehat \cR}$  coincides with the restriction to $L^\infty (M\times \Gh)^{\widehat \cR}$ 
of the extension of $\ell_{\Gamma d\gamma}$ to $L^\infty (M\times \Gh)$. 
Hence we can apply \eqref{eq_pfcor_prop_sclmR}  to the symbols $\sigma 1_{M\times \Gh_\infty}$ and $\sigma 1_{M\times \Gh_1}$ where $\sigma\in \cA_0^{\widehat \cR}$.
Furthermore, we can also apply \eqref{eq_pfcor_prop_sclmR}  to the symbol $f 1_{M\times \Gh_1}$ with $f\in \cD(M\times \fv^*)$ by Proposition \ref{prop_meascomcR} (2).
This shows Corollary \ref{cor_prop_sclmR}.

  \subsection{The case of sub-Laplacians}
  \label{subsubec_caseL}

In this section, we consider a stratified Lie group $G$, 
and fix a basis $X_1,\ldots, X_{n_1}$ of  $\fg_1$ (see Section \ref{subsec_gradedG}). 
We denote by $\cL = -X_1^2 -\ldots -X_{n_1}^2$ the associated sub-Laplacian on $G$.
This is a positive Rockland operator, 
and we can choose the first stratum $\fg_1=\bR X_1 \oplus \ldots \oplus \bR X_{n_1}$ as the complement $\fv$
of the derived algebra $[\fg,\fg] = \oplus_{i>1}\fg_i$.

In order to give a more concrete description of the objects in the previous sections in this particular case, 
we will need the following computations:

\begin{lemma}
\label{lem_compcL}
	\begin{enumerate}
		\item Decomposing an element $\omega =\sum_{j=1}^{n_1} \omega_j X_j^*$ of $\fg_1^*$ with respect to the basis dual to  $X_1,\ldots, X_{n_1}$, we have
$$
\chi_\omega(\exp (\sum_{j=1}^n x_j X_j)) = \exp(i \sum_{j=1}^{n_1} x_j \omega_j).
$$
Equipping $\fg_1$ with the scalar product that makes $X_1,\ldots, X_{n_1}$ orthonormal,  the  $\cL$-eigenvalue corresponding to $\chi_\omega$ is $|\omega|^2 = \sum_{j=1}^{n_1} \omega_j^2$:
$$
\widehat \cL \chi_\omega = |\omega|^2 \chi_\omega, 
\qquad \pi^\omega (\cL) = |\omega|^2.
$$
\item The operator $\sE$ defined in Section \ref{subsec_Qvar} is given here via:
$$
\sE = \sum_{[\alpha]=1} \Delta^\alpha \widehat \cL \ X_{\dot x}^\alpha = - 2\sum_{j=1}^{n_1} \widehat X_j X_{j,M} ,
$$
where $\widehat X_j= \{\pi(X_j): \pi\in \Gh\}$. 
In particular, $\sE$  acts on $M\times \Gh_1$ as  
$$
\sum_{[\alpha]=1} \Delta^\alpha \widehat \cL(\pi^\omega) \ X_{\dot x}^\alpha =  - 2 i
\sum_{j=1}^{n_1} \omega_j X_{j,M}.
$$
\end{enumerate}
\end{lemma}

\begin{proof}
	Part (1) is straightforward. 
	For Part (2), we can identify $G$ with $\bR^n$ via the exponential map and the choice of basis $X_j$, i.e. $(x_1,\ldots, x_n)\sim \exp_G\sum_j x_j X_j$. 
	The coordinates $q_j$ are then $q_j(x)=x_j$. 
	We compute easily for  $j=1,\ldots, n_1$:
	 $$
X_{j_1,y=0} (q_j(y)) = \delta_{j=j_1}, 
\qquad\mbox{and when}\ k\neq 1 \quad
X_{j_1,y=0}^k (q_j(y)) =0, 
$$
so 
\begin{align*}
\Delta_{q_j} \widehat X_{j_1}^2 (\pi)
&= X_{j_1,y=0}^2\left(q_j (y)\pi(y)\right)
\\&
= \left(X_{j_1,y=0}^2 q_j (y)\right) \pi(0)
+
2 \left( X_{j_1,y=0} q_j (y)\right) \left( X_{j_1,y=0}\pi(y)\right)
+
q_j (0)\left( X_{j_1,y=0}^2\pi(y)\right)
\\&= \delta_{j=j_1} 2 \widehat X_j (\pi),	
\end{align*}
therefore
$$
\Delta_{q_j} \widehat \cL (\pi) = 
- \sum_{j_1=1}^{n_1} \Delta_{q_j} \widehat X_{j_1}^2 (\pi)
= - 2 \pi(X_j)
$$
The statement follows. 	
\end{proof}

The advantage of considering sub-Laplacians is that the operation $\sE$ will lead to invariance under flows of vector fields as in the commutative case:
\begin{ex}
	We consider as in Section \ref{subsubsec_Tn} the case of the torus $\bT^n$. Then, the operator $\sE$ boils down to the vector field $-2\sum_j \partial_{\dot x_j}\partial_{\xi_j}$ on $\bT^n\times \bR^n$.
\end{ex}

Proposition \ref{prop_sclmR}, Corollary \ref{cor_prop_sclmR} and Lemma \ref{lem_compcL} yield in the case of $\cR_M=\cL_M$ the following properties of localisation and invariance for the semi-classical measures:

\begin{corollary}
\label{cor_subLinv}
	Let $(\phi_j)_{j\in \bN}$ be a sequence of eigenfunctions  for $\cL_M$ with
 $$
 \cL_M \phi_j = \mu_j  \phi_j, \quad 
 \mu_0<\mu_1\leq \mu_2\leq \ldots \leq \mu_j \longrightarrow_{j\to \infty} \infty.
 $$
   Consider a semi-classical measure $\Gamma d\gamma$ of $(\phi_j)$ at scale $\mu_j^{-1/2}$ for the subsequence $(j_k)$.
 We have for $\gamma$-almost all $(\dot x,\pi)\in M\times \Gh$
 $$
 \Gamma (\dot x,\pi) = \pi(E_1) \Gamma (\dot x,\pi) \pi(E_1).
 $$
The decomposition \eqref{eq_GdgdecGh} of $\Gamma d\gamma$ according to $\Gh = \Gh_1 \sqcup \Gh_\infty$ 
   satisfies the following properties:
  \begin{enumerate}
  	\item   The scalar valued measure
 $1_{M\times \Gh_1}\gamma$ on $M\times \Gh_1$ is supported in $M\times \{\omega\in \fg_1^* : |\omega|=1\}$
 and satisfies
 $$
\forall f\in \cD(M\times \fv^*)
\qquad
\iint_{M\times \fg_1} \sum_{j=1}^{n_1}  \omega_j  \ X_{j,M} f(\dot x,\omega)\  d \gamma(\dot x,\pi^\omega)=0.
$$
Consequently, it is invariant under the flow 
 $$
 (\dot x,\omega) \longmapsto (\exp (s \sum_{j=1}^{n_1} \omega_j X_{j,M} ) \dot x, \omega).
 $$
\item 
For $\gamma$-almost all $(\dot x, \pi)\in M\times \Gh_\infty$, the operator
  $\Gamma(\dot x,\pi)$ maps the finite dimensional 1-eigenspace for $\pi(\cL)$ onto itself and is trivial anywhere else.
  Moreover, we have  
   $$
 \iint_{M\times \Gh_\infty} {\rm Tr}\left(\widehat E_1 \sE \widehat E_1 \sigma\ \Gamma\right) d \gamma=
 \iint_{M\times \Gh_\infty} {\rm Tr}\left(\sE \sigma\ \Gamma\right) d \gamma=0,
$$
for any $\sigma\in \cA_0^{\widehat \cL}$,
where   
 $\sE = -2\sum_{j=1}^{n_1} \widehat X_j X_{j,M}$.
 \end{enumerate}	
\end{corollary}

\subsection{Comments}
\label{subsec_comments}

\subsubsection{Case of nil-manifolds of Heisenberg types}
Let us consider the particular case of Heisenberg nil-manifolds, or more generally of nil-manifolds of Heisenberg types, that is, nil-manifolds $M=\Gamma \backslash G$ with $G$ a group of Heisenberg type.

\begin{enumerate}

\item On the groups of Heisenberg type (see e.g. \cite[Appendix B]{FFJST}),
we have
$ \widehat E_1 \widehat X_j \widehat E_1=0$ for $j=1,\ldots, n_1$. 
This implies  readily on nil-manifolds  of Heisenberg type,
$$
\widehat E_1 \sE \widehat E_1
=
-2\sum_{j=1}^{n_1} \widehat E_1 \widehat X_j \widehat E_1
 \ X_{j,M} =0, 
$$
Corollary \ref{cor_subLinv}  (2) above does not give any information in this case.

	\item 
However, we can prove further invariances adapting the semi-classical analysis on groups of Heisenberg type, 
especially \cite[Lemma 4.1]{FFJST}.
This has led in \cite{FFFMilan} to further scalar invariances than the general ones described in Corollary \ref{cor_subLinv}.

In particular, by \cite[Theorem 2.4 (ii) (2b)]{FFFMilan}, 
	$1_{M\times \Gh_\infty}\Gamma d\gamma$
is invariant under $(x,\pi^\lambda)\mapsto (\exp s Z^{\lambda} x,\pi^\lambda)$, $s\in \bR$, where $\lambda$
 is a non-zero linear functional on the centre of the Lie algebra $\fg$ of $G$, $\pi^\lambda$ the corresponding representation via the orbit method, and $Z^\lambda$ the element of $\fg$ (viewed as a left-invariant vector field) corresponding to $\lambda$ by duality.
 Consequently, $\int_{\Gh_\infty} \Gamma d\gamma$ is a measure on $M$ invariant under central translation on $M$.

\item In the  case of the canonical sub-Laplacian on the Heisenberg nil-manifolds of dimension 3, a study of  quantum limits in the traditional Euclidean micro-local sense have been attempted in  \cite{letrouit}, see also \cite{CdvHT}.
Extensions beyond products of Heisenberg nil-manifolds of dimension 3
to slightly less simple case (e.g.    nil-manifolds of Heisenberg types) seem unlikely because of the increasing non-commutativity and the related increase in multiplicities of the eigenvalues of $\cL_M$.
\end{enumerate}

\subsubsection{More general cases}

\begin{enumerate}

\item The localisation and invariance given for $\Gh_1$  was obtained in \cite[Theorem 2.4]{FFFMilan}, but only for step-two nil-manifolds. Here,  Corollary \ref{cor_subLinv} (1) is proved  for any canonical sub-Laplacian on a nil-manifold. 
\item 
The localisation and invariance given for $\Gh_\infty$ obtained in \cite[Theorem 2.4]{FFFMilan} for step-two nil-manifolds is more refined than our result in Corollary \ref{cor_subLinv}.
Indeed, the analysis in \cite{FFFMilan} suggests to decompose further $\Gh_\infty$ in $\Gh = \Gh_1 \sqcup \Gh_\infty$ with the orbit method. 
Indeed, via the orbit method, $\Gh_\infty $ may be written as the disjoint union  of closed subsets $\Omega_d$ comprised of co-adjoint orbits corresponding to a given  dimension $d$ of the kernel of the associated skewsymmetric bilinear form. 
The analysis in \cite{FFFMilan} shows that 
each of these sets $\Omega_d$ will yield a different invariance on $M\times \Omega_d$.

In the case of a group of Heisenberg type, and more generally a Metivier group, $\Gh_\infty$ identifies with one such set, namely $\Omega_0$. However, for other two-step nilpotent Lie groups, $\Gh_\infty$ will decompose into various $\Omega_d$.  

\item 
It is not difficult to see that in the case when $\cR = \cL$ is a sub-Laplacian, 
the operator $\sE $ is self-adjoint on $L^2(M\times \Gh)$.
Hence the one-parameter group $e^{it\sE}$ is unitary on $L^2(M\times \Gh)$. 
It is not difficult to prove that it also acts on $\cA_0$, with furthermore $t \mapsto e^{it\sE}$ being a continuous map from $\bR$ to $\sL(\cA_0)$. 
 
In the commutative case, $M$ is a torus and it was easy to determine the kernel of $\sE$ or equivalently the subspace of $L^2(M\times \Gh)$  invariant under the action of the one-parameter group $e^{it\sE}$, see Lemma \ref{lem_thm_QETn} and its proof: it is the subspace of $\sigma\in L^2(M\times \Gh)$ such that $\int_{\bT^n} \sigma(\dot x,\pi) d\dot x =0$ for every $\pi\in \Gh$. 
In particular, $e^{it\sE} \sigma \longrightarrow_{t\infty} \int_{\bT^n} \sigma $.

In the non-commutative case, determining $\ker \sE$ or equivalently the $e^{it\sE}$-invariant subspaces of $L^2(M\times \Gh)$ is an open question as already discussed in the case of the nil-manifolds of Heisenberg type in Section \ref{subsubsec_obsgrHtype}.

%
\end{enumerate}

\end{document}